\documentclass[10pt,reqno]{amsart} 

\usepackage{amssymb,latexsym,amsmath,amsthm,amsfonts}

\begin{document}

\title{Functoriality for General Spin Groups}

\author[Mahdi Asgari]{Mahdi Asgari} 
\address{Department of Mathematics \\ 
Oklahoma State University \\ 
Stillwater, OK 74078--1058 \\
USA} 
\email{asgari@math.okstate.edu} 
	 
\author[Freydoon Shahidi]{Freydoon Shahidi} 
\address{Mathematics Department \\  
Purdue University \\ 
West Lafayette, IN 47907 \\ 
USA} 
\email{shahidi@math.purdue.edu} 

\begin{abstract}
We establish the functorial transfer of generic, automorphic representations from the 
quasi-split general spin groups to general linear groups over arbitrary number fields, 
completing an earlier project. Our results are definitive and, in particular, we determine 
the image of this transfer completely and give a number of applications. 
\end{abstract}

\maketitle

\numberwithin{equation}{section}
\newtheorem{thm}[equation]{Theorem}
\newtheorem{cor}[equation]{Corollary}
\newtheorem{lem}[equation]{Lemma}
\newtheorem{prop}[equation]{Proposition}
\newtheorem{con}[equation]{Conjecture}
\newtheorem{ass}[equation]{Assumption}
\newtheorem{defi}[equation]{Definition}
\newtheorem{exer}[equation]{Exercise}

\theoremstyle{remark}
\newtheorem{rem}[equation]{Remark}
\newtheorem{exam}[equation]{Example}

\newcommand{\A}{\mathbb A}
\newcommand{\GL}{\mathrm{GL}}
\newcommand{\GO}{\mathrm{GO}}
\newcommand{\GSO}{\mathrm{GSO}}
\newcommand{\GSp}{\mathrm{GSp}}
\newcommand{\GSpin}{\mathrm{GSpin}}
\newcommand{\SO}{\mathrm{SO}}
\newcommand{\Spin}{\mathrm{Spin}}

\newcommand{\C}{\mathbb C} 
\newcommand{\Q}{\mathbb Q} 
\newcommand{\Z}{\mathbb Z} 
\newcommand{\R}{\mathbb R} 

\newcommand{\B}{{\bf B}} 
\newcommand{\G}{{\bf G}} 
\renewcommand{\H}{{\bf H}} 
\newcommand{\M}{{\bf M}} 
\newcommand{\N}{{\bf N}} 
\renewcommand{\P}{{\bf P}} 
\newcommand{\T}{{\bf T}} 
\newcommand{\U}{{\bf U}} 
\newcommand{\X}{{\bf X}} 
\newcommand{\Y}{{\bf Y}} 
\newcommand{\ZZ}{{\bf Z}} 
  
\newcommand{\sym}{{\rm Sym}} 

\renewcommand{\i}{\mathfrak i}

\newcommand{\w}[1]{\ensuremath{\widetilde{#1}}} 
\newcommand{\h}[1]{\ensuremath{\widehat{#1}}} 

\section{Introduction}\label{intro} 

In this article we complete a project we started in \cite{duke} by establishing the 
full transfer of generic, automorphic representations from the quasi-split general 
spin groups to the general linear group. In particular, we completely determine the 
image of this transfer. 

Our first main result is to establish the transfer of globally generic, automorphic representations 
from the quasi-split non-split even general spin group, $\GSpin^{*}(2n,\A_{k})$, to 
$\GL(2n,\A_{k})$ (cf. Theorem \ref{q-split-weak}). Here, $k$ denotes an 
arbitrary number field. We proved the analogous result for 
the split groups $\GSpin(2n)$ and $\GSpin(2n+1)$ in \cite{duke}, but were not able 
to prove the quasi-split case then since the ``stability of root numbers'' was not 
yet available for non-split groups.  

Our next main result is to prove that the transferred representation to $\GL(2n)$, 
from either an even or an odd general spin group, is actually an isobaric, automorphic 
representation (cf. Theorem \ref{image}). 

Our final result gives a complete description of the image of 
this transfer in terms of $L$-functions (cf. Theorem \ref{main}). This description 
is exactly what is expected from the theory of twisted endoscopy. 

The latter two results allow us to give a number of applications. 
As a first application, we are able to describe the local component of the 
transferred representation at the ramified places. In particular, we show that 
these local components are generic (cf. Proposition \ref{vinS}). 

Another application is to prove estimates toward the Ramanujan conjecture for 
the generic spectrum of the general spin groups. We do this by using the best estimates 
currently known for the general linear groups \cite{lrs}. In particular, our estimates show that 
if we know the Ramanujan conjecture for $\GL(m)$ for $m$ up to $2n$, then the 
Ramanujan conjecture for the generic spectrum of $\GSpin(2n+1)$ and $\GSpin(2n)$ 
follows. 

Yet another application of our main results is to give more information about H. Kim's 
exterior square transfer from $\GL(4)$ to $\GL(6)$ with the help of some 
recent work of J. Hundley and E. Sayag. We prove that a cuspidal 
representation $\Pi$ of $\GL(6)$ is in the image of Kim's transfer if and only if 
the (partial) twisted symmetric square $L$-function of $\Pi$ has a pole at 
$s=1$ (cf. Proposition \ref{ext2}). 

We now explain our results in more detail. 
Let $k$ be a number field and let $\A=\A_{k}$ denote its ring of ad\`eles. 
Let $\G$ be the split group $\GSpin(2n+1)$, $\GSpin(2n)$ or one of its quasi-split non-split 
forms $\GSpin^{*}(2n)$ associated with a quadratic extension $K/k$ (cf. Section \ref{pre-sec}). 
There is a natural embedding 
\begin{equation} 
\iota : {}^{L}\G \longrightarrow \GL(2n,\C) \times \Gamma_k
\end{equation} 
of the $L$-group of $\G$, as a group over $k$, into that of $\GL(2n)$ described in Section \ref{q-split-sec}. 
Let $\pi$ be a globally generic, (unitary) cuspidal, automorphic representation of $G=\G(\A)$. 
For almost all places $v$ of $k$ the local representation $\pi_{v}$ is parametrized by 
a homomorphism 
\begin{equation} 
\phi_{v} : W_{v} \longrightarrow {}^{L}\G_v, 
\end{equation} 
where $W_{v}$ is the local Weil group of $k_{v}$ and ${}^L\G_v$ is the $L$-group of $\G$ as a group 
over $k_v$. Langlands Functoriality then predicts that there is an automorphic representation 
$\Pi$ of $\GL(2n,\A)$ such that for almost all $v$, the local representation $\Pi_{v}$ 
is parametrized by $\iota \circ \phi_{v}$. 
We established this result for the split case in \cite{duke}. However, the quasi-split case had to 
wait because the local technical tools of ``stability of $\gamma$-factors'' (cf. Proposition \ref{stab}) 
and a result on local $L$-functions and normalized intertwining operators (Proposition \ref{nice}) 
were not available in the quasi-split non-split case. The local result is now available in our cases 
thanks to the thesis of Wook Kim \cite{w-kim} and the stability of $\gamma$-factors is available in great 
generality thanks to a recent work of Cogdell, Piatetski-Shapiro and Shahidi \cite{cpss-stab}. 

As in the split case, the method of proving the existence of an automorphic representation 
$\Pi$ is to use converse theorems. This requires knowledge of the analytic properties of the 
$L$-functions for $\GL(m) \times \GL(2n)$ for $m \le 2n-1$.  
The two local tools allow us to relate the $L$-functions for $\G \times \GL$ from the 
Langlands-Shahidi method to those required in the converse theorems in the following way. 
Due to the lack of the local Langlands correspondence in general, there is no natural choice 
for the local components of our candidate 
representation $\Pi$ at the finite number of exceptional places of $k$ where some of our 
data may be ramified. This means that we have to pick these local 
representations essentially arbitrarily. However, we show that the local $L$- and $\epsilon$-factors 
appearing will become independent of the representation, depending only on the central character, 
if we twist by a highly ramified character. Globally we can afford to twist our original representation 
by an id\`ele class character which is highly ramified at a finite number of places. With this 
technique we succeed in applying an appropriate version of the converse theorem. The conclusion 
so far is to have an automorphic representation $\Pi$ of $\GL(2n,\A)$ which is locally 
the transfer of $\pi$ associated with $\iota$ outside a finite number of places. Moreover, 
if $\omega=\omega_{\pi}$ is the central character of $\pi$, then $\omega_{\Pi} = \omega^{n} \mu$, 
where $\mu$ is a quadratic id\`ele class character, only nontrivial in the quasi-split non-split case. 

Next, we get more information about $\Pi$. In particular, we prove that $\Pi$ is indeed an 
isobaric, automorphic representation (cf. Theorem \ref{image} and its corollary). We refer to 
\cite{langlands} for the notion of isobaric representations. 
For this one needs to know some analytic properties of 
the Rankin-Selberg type $L$-functions $L(s, \pi \times \tau)$, where $\tau$ is a cuspidal 
representation of $\GL(m,\A)$ and $\pi$ is a generic representation of $\G(\A)$. In particular, 
one needs to know that the $L$-function for $\G \times \GL(m)$ for $m \le n$ is holomorphic 
for $\Re(s) > 1$ and to know under what conditions this $L$-function has a pole at $s=1$. 
When $\G$ is a special orthogonal or symplectic group, these results are known thanks to 
the works of Gelbart, Ginzburg, Piatetski-Shapiro, Rallis and Soudry studying the Rankin-Selberg type 
zeta integrals giving these $L$-functions. For a survey of the methods and results we refer to 
Soudry's survey article \cite{soudry}. In Section \ref{grs} we extend this method to the 
case of $\G \times \GL$, where $\G$ is a general spin group, closely following their method. 

Not only does the study of the Rankin-Selberg integral gives us the result on the analytic 
behavior of the $L$-functions, it also provides more information 
about the image of the transfer. In particular, we prove that $\Pi \cong \w{\Pi} \otimes \omega$. 
Here, $\w{\Pi}$ denotes the contragredient of $\Pi$. 
We are also able to prove that the transferred 
representation $\Pi$ is unique, it is an isobaric sum of pairwise inequivalent unitary, cuspidal, 
automorphic representations 
\begin{equation} \Pi = \Pi_{1} \boxplus \cdots \boxplus \Pi_{t}, \end{equation}
and each $\Pi_{i}$ satisfies the condition that its twisted symmetric square or twisted exterior square 
$L$-function has a pole at $s=1$, depending on whether we are transferring from even or odd 
general spin groups (cf. Theorem \ref{main}). 

The automorphic representations $\Pi$ of $\GL(2n,\A)$ which are transfers from representations 
$\pi$ of general spin groups satisfy 
\begin{equation} \label{omega}
\Pi \cong \w{\Pi} \otimes \omega, 
\end{equation}
as predicted by the theory of twisted endoscopy \cite{kottwitz-shelstad}. In fact, these representations 
comprise precisely the image of the transfer. While we prove half of this statement we note that 
the other half of this, i.e., the fact that any representation of $\GL(2n,\A)$ satisfying (\ref{omega}) 
is a transfer from a representation of a general spin group has now also been proved thanks to 
the work of J. Hundley and E. Sayag, extending the descent theory results of 
Ginzburg, Rallis, and Soudry from the case of classical groups ($\omega=1$) to our case. 

If a representation $\Pi$ of $\GL(2n,\A)$ satisfies (\ref{omega}), then 
\begin{equation} \label{L-pole}
L^{T}\left(s, \Pi \times (\Pi\otimes\omega^{-1})\right) = 
L^{T}(s,\Pi,\sym^{2}\otimes\omega^{-1})
L^{T}(s,\Pi,\wedge^{2}\otimes\omega^{-1}), 
\end{equation}
where $T$ is a sufficiently large finite set of places of $k$ and $L^{T}$ denotes the product 
over $v \not\in T$ of the local $L$-functions. 
The $L$-function on the left hand side of (\ref{L-pole}) has a pole at $s=1$, which implies 
that one, and only one, of the two $L$-functions on the right hand side of (\ref{L-pole}) 
has a pole at $s=1$. If the twisted exterior square $L$-function has a pole at $s=1$, 
then $\Pi$ is a transfer from an odd general spin group and if the twisted symmetric 
square $L$-function has a pole at $s=1$, then $\Pi$ is a transfer from an even general 
spin group (which may be split or quasi-split non-split). 

To tell the split and quasi-split non-split cases apart note that from (\ref{omega}) we have 
\begin{equation} 
\omega_{\Pi}^{2} = \omega^{2n}. 
\end{equation} 
In other words, $\mu = \omega_{\Pi} \omega^{-n}$ is a 
quadratic id\`ele class character. If $\mu$ is the trivial character, then $\Pi$ is the transfer of 
a generic representation of even split general spin group and if $\mu$ is a nontrivial quadratic 
character, then $\Pi$ is a transfer from a generic representation of a quasi-split group associated 
with the quadratic extension of $k$ determined by $\mu$ through class field theory. 

Our results here along with those of Hundley and Sayag \cite{hs-announce,hs-ppt-odd,hs-ppt-even} 
give a complete description of the 
image of the transfer for the generic representations of the general spin groups. It remains 
to study the transfers of non-generic, cuspidal, automorphic representations of the general spin 
groups, which our current method cannot handle. However, the image of the generic 
spectrum is conjecturally the full image of the tempered spectrum, generic or not, of the 
general spin groups since each tempered $L$-packet is expected to include a generic 
member. We refer to \cite{shahidi-ppt} for more details on this conjecture. We point out that 
Arthur's upcoming book \cite{arthur} would answer this question in the case of classical 
groups. However, his book does not cover the case of general spin groups. 

We can apply our results in this paper, along with those of Cogdell, Kim, Krishnamurthy, 
Piatetski-Shapiro, and Shahidi for the classical and unitary groups, to give some uniform 
results on reducibility of local induced representations of non-exceptional $p$-adic groups. 
We will address this question along with other local applications of generic functoriality 
in a forthcoming paper \cite{me-cogdell-shahidi}. 

We would like to thank J. Cogdell, J. Hundley, E. Sayag and D. Soudry for helpful discussions. 
Parts of this work was done while the authors were visiting The Erwin Schr\"odinger Institute 
for Mathematical Physics (ESI) and the Institute for Research in Fundamental Sciences (IPM). 
We would like to thank both institutes for their hospitality and great working environment. 
The first author was partially supported by NSA grant H-98230--09--1--0049 and an 
Alexander von Humboldt Fellowship. He also acknowledges partial travel support from 
Vaughn funds at Oklahoma State University. The second author was partially supported 
by NSF grant DMS-0700280.

\tableofcontents

\section{The Preliminaries} \label{pre-sec}

Let $k$ be a number field and let $\A=\A_k$ be the ring of ad\`eles of $k$. Let 
$n \ge 0$ be an integer. We consider the general spin groups. The group 
$\GSpin(2n+1)$ is a split connected reductive group of type $B_{n}$ defined 
over $k$ whose dual group is $\GSp(2n,\C)$. 
Similarly, the split connected reductive group $\GSpin(2n)$ over $k$ is of type $D_n$ 
and its dual is isomorphic to $\GSO(2n,\C)$, the connected component of the group $\GO(2n,\C)$. 
There are also quasi-split non-split groups $\GSpin^{*}(2n)$ in the even case. 
They are of type $^{2}D_{n}$ and correspond to quadratic extensions of $k$. 
A more precise description is given below. We also refer to \cite[\S 7 \& 1]{cpss-clay} 
for a review of the generalities about these groups. 

We fix a Borel subgroup $\B$ and a Cartan subgroup $\T \subset \B$. The associated 
based root datum to $(\B,\T)$ will be denoted by $(X,\Delta,X^{\vee},\Delta^{\vee})$ 
which we further explicate below. Our choice of the notation for the 
root data below is consistent with the Bourbaki notation \cite{bourbaki}.

\subsection{Structure of $\GSpin$ Groups}
We describe the odd and even $\GSpin$ groups by introducing 
a based root datum for each as in \cite[\S 7.4.1]{spr}. A more detailed description can also 
be found in \cite[\S 2]{duke}. We use these data as our tool to work with the groups in 
question due to the lack of a convenient matrix representation.

\subsubsection{The root datum of $\GSpin(2n+1)$}\label{G-odd} 
The root datum of $\GSpin(2n+1)$ is given by $(X,R,X^{\vee},R^{\vee})$, where 
$X$ and $X^{\vee}$ are $\mathbb Z$-modules generated by generators $e_{0},e_{1},\cdots,e_{n}$ 
and $e^{*}_{0},e^{*}_{1},\cdots,e^{*}_{n}$, respectively. The pairing 
\begin{equation}\label{pairing} 
\langle \ , \ \rangle : X \times X^{\vee} \longrightarrow \mathbb Z
\end{equation}
is the standard pairing. Moreover, the roots and coroots are given by 
\begin{eqnarray}
\label{R-odd}
R = R_{2n+1} &=& \left\{ \pm (e_{i} \pm e_{j}) \, : \, 1 \le i<j \le n \right\} \cup 
\left\{\pm e_{i} \, | \, 1\le i \le n \right\} \\
\label{R-odd-co}
R^{\vee} = R^{\vee}_{2n+1} &=& \left\{ \pm(e^{*}_{i} - e^{*}_{j}) \, : \, 1 \le i<j \le n \right\} \cup \\ 
&&\nonumber 
\left\{ \pm(e^{*}_{i} + e^{*}_{j} - e_0^*) \, : \, 1 \le i<j \le n \right\} \cup 
\left\{\pm (2e^{*}_{i} - e^{*}_{0}) \, | \, 1\le i \le n \right\} 
\end{eqnarray}
along with the bijection $R \longrightarrow R^{\vee}$ given by 
\begin{eqnarray}
\label{v-odd-1}
(\pm(e_{i} - e_{j}))^{\vee} &=&  \pm (e^{*}_{i} - e^{*}_{j}) \\ 
\label{v-odd-2}
(\pm(e_{i} + e_{j}))^{\vee} &=&  \pm (e^{*}_{i} + e^{*}_{j} - e^{*}_{0}) \\ 
\label{v-odd-3}
(\pm e_{i})^{\vee} &=& \pm (2e^{*}_{i} - e^{*}_{0}).  
\end{eqnarray} 
It is easy to verify that the conditions (RD 1) and (RD 2) of \cite[\S 7.4.1]{spr} hold. 
Moreover, we fix the following choice of simple roots and coroots: 
\begin{eqnarray}
\Delta &=& \left\{ e_1-e_2, e_2-e_3,\cdots,e_{n-1}-e_n, e_n \right\}, \\ 
\Delta^\vee &=& \left\{ e^*_1-e^*_2, e^*_2-e^*_3,\cdots,e^*_{n-1}-e^*_n, 2e^*_n-e^*_0 \right\}.
\end{eqnarray} 
This datum determines the group $\GSpin(2n+1)$ uniquely, equipped with a Borel subgroup 
containing a maximal torus. 

\subsubsection{The root datum of $\GSpin(2n)$}\label{G-even} 
Next, we give a similar description for the even case. 
The root datum of $\GSpin(2n)$ is given by $(X,R,X^{\vee},R^{\vee})$ where 
$X$ and $X^{\vee}$ and the pairing is as above and the roots and coroots are 
given by 
\begin{eqnarray}
\label{R-even}
R = R_{2n} &=& \left\{ \pm (e_{i} \pm e_{j}) \, : \, 1 \le i < j \le n \right\}  \\
\label{R-even-co}
R^{\vee} = R^{\vee}_{2n} &=& \left\{ \pm (e^{*}_{i} - e^{*}_{j}) \, : \, 1 \le i<j \le n \right\} \cup \\ 
\nonumber
&& \left\{ \pm (e^{*}_{i} + e^{*}_{j} - e^{*}_{0}) \, : \, 1 \le i<j \le n \right\} 
\end{eqnarray}
along with the bijection $R \longrightarrow R^{\vee}$ given by 
\begin{eqnarray}
\label{v-even-1}
(\pm(e_{i} - e_{j}))^{\vee} &=&  \pm (e^{*}_{i}-e^{*}_{j}) \\ 
\label{v-even-2}
(\pm (e_{i} + e_{j}))^{\vee} &=& \pm (e^{*}_{i} + e^{*}_{j} - e^{*}_{0}).  
\end{eqnarray} 
It is easy again to verify that the conditions (RD 1) and (RD 2) of \cite[\S 7.4.1]{spr} hold. 
Similar to the odd case we fix the following choice of simple roots and coroots: 
\begin{eqnarray}
\Delta &=& \left\{ e_1-e_2, e_2-e_3,\cdots,e_{n-1}-e_n, e_{n-1}+e_n \right\}, \\ 
\Delta^\vee &=& \left\{ e^*_1-e^*_2, e^*_2-e^*_3,\cdots,e^*_{n-1}-e^*_n, e^{*}_{n-1}+e^{*}_{n}-e^{*}_{0} \right\}.
\end{eqnarray} 
This based root datum determines the group $\GSpin(2n)$ uniquely, equipped with a Borel subgroup 
containing a maximal torus. 

\subsubsection{The quasi-split forms of $\GSpin(2n)$}\label{G-even-q-split} 
In the even case, quasi-split non-split forms also exist. We fix a splitting 
$(\B, \T, \{x_{\alpha}\}_{\alpha\in\Delta})$, where $\{x_{\alpha}\}$ is a collection of root vectors, 
one for each simple root of $\T$ in $\B$. As explained in \cite[\S 7.1]{cpss-clay} for 
the quasi-split forms of $\SO(2n)$, the quasi-split forms of $\GSpin(2n)$ over $k$ 
are in one-one correspondence with homomorphisms from $\mbox{Gal}(\bar{k}/k)$ to the group of 
automorphisms of the character lattice preserving $\Delta$. This group has two elements: the trivial 
and the one switching $e_{n-1}-e_n$ and $e_{n-1}+e_n$ while keeping all other simple roots fixed. 

By Class Field Theory such homomorphisms correspond to quadratic characters 
\[ \mu : k^\times \backslash \A_k^\times \longrightarrow \{\pm 1\}. \]
When $\mu$ is nontrivial we denote the associated quasi-split non-split group with $\GSpin^\mu(2n)$ 
or simply $\GSpin^*(2n)$ when the particular $\mu$ is unimportant. We will also denote the 
quadratic extension of $k$ associated with $\mu$ by $K^\mu/K$ or simply $K/k$.

\subsection{Embeddings} \label{embed}

In order to proceed with the analogues of the relevant integrals 
for $\GSpin$ groups we need 
certain embeddings of these groups inside each other, which we now review. We first recall some 
basic facts about linear algebraic groups. 

Let $\G$ be a connected reductive linear algebraic group over $k$ with a fixed Borel 
subgroup containing a fixed maximal torus.  Denote the associated roots by $R$ and 
the positive roots by $R^+$. For each $\alpha \in R$ denote the root group 
homomorphism associated with $\alpha$ by  
\[ u_\alpha : \mathbb{G}_a \longrightarrow \G \]
and denote the root group by $\U_\alpha$, the image of $u_\alpha$ in $\G$. 
We are now prepared to describe the various embeddings we will need. 
 
\subsubsection{Embedding $\i : \GSpin(2m+1) \hookrightarrow \GSpin(2n)$ for $m < n$} 
The group $\GSpin(2m+1)$ is an {\it almost direct} product (i.e., with finite intersection) of 
$\Spin(2m+1)$, the derived group of $\GSpin(2m+1)$, and the connected component of its 
center, which is a torus. The derived group is generated by the root subgroups and, by
\cite[Prop. 2.3]{duke}, the connected component of the center equals  
\[ \left\{e^{*}_{0}(t) \,|\, t \in \GL(1) \right\}. \]

To describe the embedding we embed each root subgroup of $\GSpin(2m+1)$ into $\GSpin(2n)$ 
and also embed the connected component of the center of the former into the latter. 
We should, however, ensure that the images of elements in the intersection are consistent. 

Let us use the notation $e_i$ and $e^*_i$ for roots and coroots of $\GSpin(2m+1)$ and 
$f_i$ and $f^*_i$ for those of $\GSpin(2n)$. 
By (\ref{R-odd}) the roots of $\GSpin(2m+1)$ are given by 
\[ R_{2m+1} = \left\{ \pm e_i \, | \, 1 \le i \le m \right\} \cup \left\{ \pm (e_i \pm e_j) \, | \, 1 \le i < j \le m  \right\}\] 
and by (\ref{R-even}) those of $\GSpin(2n)$ may be written as 
\[ R_{2n} = \left\{ \pm (f_i \pm f_j) \, | \, 1 \le i < j \le n  \right\}.\] 

For $1 \le i<j \le m,$ we define 
\begin{eqnarray}
\i\left( u_{e_i - e_j}(x) \right) &=& u_{f_i - f_j}(x) \label{i-odd-1}\\ 
\i\left( u_{e_i + e_j}(x) \right) &=& u_{f_i + f_j}(x) \label{i-odd-2}  
\end{eqnarray}
and for $1 \le i \le m,$ we define 
\begin{equation}
\i\left(u_{e_i}(x) \right) = u_{f_i - f_{n}}(x) u_{f_i + f_{n}}(-x).  \label{i-odd-3}
\end{equation}
For negative roots we define $\i$ in a similar way using the corresponding negative roots on 
the right hand side. Also, set 
\begin{equation}
\i\left(e^*_0(t) \right) = f^*_0(t).  \label{i-odd-4}
\end{equation}

\begin{lem}
For $m<n$ the embedding $\i:\GSpin(2m+1) \hookrightarrow \GSpin(2n)$ 
defined via (\ref{i-odd-1})--(\ref{i-odd-4}) is well-defined. 
\end{lem}

\begin{proof}
We have to check that $\i$ is well-defined on the intersection of the derived group and 
the connected component of the center. 

As verified in \cite[\S 2]{duke} the intersection consists of the trivial element and 
the nontrivial element $c=\alpha^\vee_m(-1)$. On the one hand we have 
\[ \alpha^\vee_m(-1) = (2 e^{*}_{m}-e^{*}_{0})(-1) = e^{*}_{0}(-1)^{-1}=e^{*}_{0}(-1) \] 
and consequently
\[ \i(c) = f^{*}_{0}(-1). \] 

On the other hand $c$ also belongs to the derived group and can be expressed in terms of 
the root group homomorphisms as follows. Recall \cite[Eq. (18)--(20)]{duke} 
that for any root $\alpha$ we have 
\begin{equation} 
\alpha^\vee(-1) = w_\alpha^2 \label{w-alpha}, 
\end{equation} 
where $w_\alpha = u_\alpha(1) u_{-\alpha}(-1) u_\alpha(1)$. 
This means that 
\begin{eqnarray*} 
\i\left(w_{\alpha_{m}}\right) &=& \i\left( u_{e_{m}}(1) u_{-e_{m}}(-1) u_{e_{m}}(1) \right) \\
&=&  
u_{f_{m}-f_{n}}(1) u_{f_{m}+f_{n}}(-1) \cdot
u_{-f_{m}+f_{n}}(-1) u_{-f_{m}-f_{n}}(1) \cdot
u_{f_{m}-f_{n}}(1) u_{f_{m}+f_{n}}(-1)   \\ 
&=& 
u_{f_{m}-f_{n}}(1) u_{-f_{m}+f_{n}}(-1) u_{f_{m}-f_{n}}(1) \cdot 
u_{f_{m}+f_{n}}(-1) u_{-f_{m}-f_{n}}(1) u_{f_{m}+f_{n}}(-1)  \\  
&=& 
w_{f_{m}-f_{n}} w_{f_{m}+f_{n}}^{-1}.  
\end{eqnarray*} 
Hence, 
\begin{eqnarray*}
\i(c) &=& \i(w_{\alpha_{m}})^{2} \\ 
&=& ( w_{f_{m}-f_{n}} w_{f_{m}+f_{n}}^{-1} )^{2} \\
&=& (f_{m}-f_{n})^{\vee}(-1) \cdot (-f_{m}-f_{n})^{\vee}(-1) \\ 
&=& (f^{*}_{m}-f^{*}_{n}-f^{*}_{m}-f^{*}_{n}+f^{*}_{0})(-1) \\ 
&=& f^{*}_{0} (-1). 
\end{eqnarray*}
Here, we have used (\ref{v-even-1}), (\ref{v-even-2}), and the fact that 
$w_{f_{m}-f_{n}}$ and  $w_{f_{m}+f_{n}}$ commute \cite[p. 157]{duke}.
We conclude that $\i(c)$ is well-defined and this proves the lemma. 
\end{proof} 

\subsubsection{Embedding $\i : \GSpin(2m) \hookrightarrow \GSpin(2n+1)$ for $m \le n$} %
We proceed in a similar way for this embedding as well. Recall that again the group 
$\GSpin(2m)$ is an almost direct product of its derived group and the connected 
component of the center
\[ \left\{e^{*}_{0}(t) \,|\, t \in \GL(1) \right\}.\]

Using a similar notation as before recall that as in (\ref{R-odd}) and (\ref{R-even}) the roots 
of $\GSpin(2m)$ and $\GSpin(2n+1)$ are respectively given by 
\[ R_{2m} = \left\{ \pm (e_i \pm e_j) \, | \, 1 \le i < j \le m  \right\}.\] 
and 
\[ R_{2n+1} = \left\{ \pm f_i \, | \, 1 \le i \le n \right\} \cup 
\left\{ \pm (f_i \pm f_j) \, | \, 1 \le i < j \le n  \right\}.\] 

For $1 \le i<j \le m$ we define 
\begin{eqnarray}
\i\left( u_{e_i - e_j}(x) \right) &=& u_{f_i - f_j}(x) \label{i-even-1}\\ 
\i\left( u_{e_i + e_j}(x) \right) &=& u_{f_i + f_j}(x) \label{i-even-2}.  
\end{eqnarray}
For negative roots we define $i$ in a similar way using the corresponding negative roots on 
the right hand side. Also, set 
\begin{equation}
\i\left(e^*_0(t) \right) = f^*_0(t).  \label{i-even-3}
\end{equation}

\begin{lem}
For $m \le n$ the embedding $\i:\GSpin(2m) \hookrightarrow \GSpin(2n+1)$ 
defined via (\ref{i-even-1})--(\ref{i-even-3}) is well-defined. 
\end{lem}

\begin{proof}
We have to check that $\i$ is well-defined on the intersection of the derived group and 
the connected component of the center. This intersection again consists of two elements 
with the nontrivial element now being 
$c=\alpha^\vee_{m-1}(-1)\alpha^\vee_m(-1)$ \cite[\S 2]{duke}. 

On the one hand we have 
\[ \alpha^\vee_{m-1}(-1)\alpha^\vee_m(-1) = (2 e^{*}_{m-1}-e^{*}_{0})(-1) = e^{*}_{0}(-1)^{-1}=e^{*}_{0}(-1) \] 
and consequently 
\[ \i(c) = f^{*}_{0}(-1). \] 

On the other hand $c$ also belongs to the derived group. We have 
\begin{eqnarray*} 
\i\left(w_{\alpha_{m-1}} w_{\alpha_{m}}\right) &=& 
\i\left( u_{e_{m-1}-e_{m}}(1) u_{-e_{m-1}+e_{m}}(-1) u_{e_{m-1}-e_{m}}(1) \right.  \\ 
&& \quad \cdot \left. u_{e_{m-1}+e_{m}}(1) u_{-e_{m-1}-e_{m}}(-1) u_{e_{m-1}-e_{m}}(1) \right) \\
&=&  
u_{f_{m-1}-f_{m}}(1) u_{-f_{m-1}+f_{m}}(-1) u_{f_{m-1}-f_{m}}(1)  \\ 
&& \cdot \,  u_{f_{m-1}+f_{m}}(1) u_{-f_{m-1}-f_{m}}(-1) u_{f_{m-1}-f_{m}}(1) \\ 
&=& 
w_{f_{m-1}-f_{m}} w_{f_{m-1}+f_{m}}.  
\end{eqnarray*} 
Hence,                  
\begin{eqnarray*}
\i(c) &=& \i(  w_{\alpha_{m-1}}^{2} w_{\alpha_{m}}^{2} ) \\ 
&=& (f_{m-1}-f_{m})^{\vee}(-1) \cdot (f_{m-1}+f_{m})^{\vee}(-1) \\ 
&=& (f^{*}_{m-1}-f^{*}_{m}+f^{*}_{m-1}+f^{*}_{m}-f^{*}_{0})(-1) \\ 
&=& f^{*}_{0} (-1). 
\end{eqnarray*}
Here, we have used (\ref{v-odd-1})--(\ref{v-odd-3}), and the fact that 
$w_{f_{m}-f_{n}}$ and  $w_{f_{m}+f_{n}}$ commute \cite[p. 157]{duke}.
We conclude that $\i(c)$ is well-defined and this proves the lemma. 
\end{proof} 

We end this section by proving a lemma which explicitly gives the image under the map $\i$ 
of elements in the maximal torus of $\H$.  We will use this lemma later when we discuss 
convergence of certain zeta integrals we have to deal with. 

\begin{lem}\label{torus-i}
With notation as above, let $a = e_0^*(t_0) e_1^*(t_1) \cdots e_m^*(t_m)$ be an arbitrary 
element in the maximal torus of $\H$. Then, 
\[ \i(a) = \begin{cases}
f_0^*(t_0) f_1^*(t_1) \cdots f_{m-1}^*(t_{m-1}) f_m^*(t_m) 
& \text{ if $\H$ is even and $\G$ is odd}, \\
\\
f_0^*(-t_0) f_1^*(t_1) \cdots f_{m-1}^*(t_{m-1}) f_m^*(-t_m)f_n^*(-1) 
& \text{ if $\H$ is odd and $\G$ is even.}
\end{cases} \] 
\end{lem}

\begin{proof}
The proof is essentially a careful chasing of the definitions in terms of the root data. 
First, assume that $\H=\GSpin(2m).$ We have 
\begin{eqnarray*}
a &=& e_0^*\left(t_0 (t_1 \cdots t_{m-1} t_m)^{1/2}\right) 
(e_1^*-e_2^*)(t_1) (e_2^*-e_3^*)(t_1 t_2) \cdots (e_{m-2}^*-e_{m-1}^*)\left(t_1 t_2 \cdots t_{m-2}\right) \\  
&& (e_{m-1}^*-e_m^*)\left((t_1 \cdots t_{m-1} t_m^{-1})^{1/2}\right) 
(e_{m-1}^*+e_m^*-e_0^*)\left((t_1 \cdots t_{m-1} t_m)^{1/2}\right),   
\end{eqnarray*}
where the choices of the square roots have to be made appropriately. More precisely, in order to get 
$e_{m-1}^*(t_{m-1})$ and $e_m^*(t_m)$ we have to choose the square roots in the last two terms 
consistently. Writing $( t_1 \cdots t_{m-1} t_m^{-1} )^{1/2} = ( t_1 \cdots t_{m-1} )^{1/2} \, t_m^{-1/2}$, we make 
arbitrary choices for the square roots $( t_1 \cdots t_{m-1} )^{1/2}$ and $t_m^{-1/2}$ in the term $e_{m-1}^*-e_m^*$ and 
use the same choices of the square roots for $( t_1 \cdots t_{m-1} )^{1/2} \, t_m^{1/2} = (t_1 \cdots t_{m-1} t_m)^{1/2}$ 
in the last term $e_{m-1}^*+e_m^*-e_0^*$. Now, in order to get $e_0^*(t_0)$ the choice of the square root in the first 
term $e_0^*$ has to be the same as that in the last term. 

Applying the map $\i$ and using the root data details we get 
\begin{eqnarray*}
\i(a) &=& f_0^*\left(t_0 (t_1 \cdots t_{m-1} t_m)^{1/2}\right) (f_1-f_2)^\vee(t_1)\cdots
(f_{m-2}-f_{m-1})^\vee(t_1\cdots t_{m-1}) \\ 
&& (f_{m-1}-f_m)^\vee\left((t_1 \cdots t_{m-1} t_m^{-1})^{1/2}\right) 
(f_{m-1}+f_m)^\vee\left((t_1 \cdots t_{m-1} t_m)^{1/2}\right) \\
&=& f_0^*\left(t_0 (t_1 \cdots t_{m-1} t_m)^{1/2}\right) 
(f_1^*-f_2^*)(t_1) \cdots (f_{m-2}^*-f_{m-1}^*)\left(t_1 t_2 \cdots t_{m-2}\right) \\
&& (f_{m-1}^*-f_m^*)\left((t_1 \cdots t_{m-1} t_m^{-1})^{1/2}\right) 
(f_{m-1}^*+f_m^*-f_0^*)\left((t_1 \cdots t_{m-1} t_m)^{1/2}\right) \\
&=& f_0^*(t_0) f_1^*(t_1) \cdots f_{m-1}^*(t_{m-1}) f_m^*(t_m)
\end{eqnarray*} 
Again, similar consistent choices of the square roots need to be made.  

Next, assume that $\H=\GSpin(2m+1).$ A similar calculation will go through. A necessary step, however, 
is to calculate the image under $\i$ of $e_m^\vee(x)$. To do this, note that 
\[ e_m^\vee(x) = w_{e_m}(x) w_{e_m}(1)^{-1} = u_{e_m}(x) u_{-e_m}(-x^{-1}) u_{e_m}(x) 
\left(u_{e_m}(1) u_{-e_m}(-1) u_{e_m}(1) \right)^{-1}. \] 
Apply $\i$ to get 
\begin{eqnarray*}
\i(e_m^\vee(x)) &=& 
u_{f_m-f_n}(x) u_{f_m+f_n}(-x) \cdot u_{-f_m+f_n}(-x^{-1}) u_{-f_m-f_n}(x^{-1}) \cdot u_{f_m-f_n}(x) u_{f_m+f_n}(-x) \\ 
&& 
\left(u_{f_m-f_n}(1) u_{f_m+f_n}(-1) \cdot u_{-f_m+f_n}(-1) u_{-f_m-f_n}(1) \cdot u_{f_m-f_n}(1) u_{f_m+f_n}(-1)\right)^{-1} \\ 
&=&
u_{f_m-f_n}(x) u_{-f_m+f_n}(-x^{-1}) u_{f_m-f_n}(x) \left( u_{f_m-f_n}(1) u_{-f_m+f_n}(-1) u_{f_m-f_n}(1) \right)^{-1} \\ 
&& 
u_{f_m+f_n}(-x) u_{-f_m-f_n}(x^{-1}) u_{f_m+f_n}(-x) \left( u_{f_m+f_n}(-1) u_{-f_m-f_n}(1) u_{f_m+f_n}(-1) \right)^{-1} \\ 
&=& w_{f_m-f_n}(x) w_{f_m-f_n}(1)^{-1} \cdot w_{f_m+f_n}(-x) w_{f_m+f_n}(-1)^{-1} \\
&=& \left(f_m-f_n\right)^\vee(x) \cdot \left(f_m+f_n\right)^\vee(-x). 
\end{eqnarray*} 
This last calculation is responsible for the appearance of the negative signs in the $\H$ odd case in the 
statement of the lemma. 
\end{proof}

\section{Weak Transfer for the Quasi-split $\GSpin(2n)$} \label{q-split-sec}

In this section $\G=\GSpin^*(2n)$ will denote one of the quasi-split non-split forms of 
$\GSpin(2n)$ as in \ref{G-even-q-split}. We will denote the associated quadratic 
extension by $K/k$ and $\A=\A_k$ will continue to denote the ring of ad\`eles of 
$k$. Also, $\G$ is associated with a nontrivial quadratic character 
$\mu : k^\times \backslash \A_k^\times \longrightarrow \{\pm 1\}$. 
 
The connected component of the $L$-group of $\G$ is 
${}^L\G^0 = \GSO(2n,\C)$ and the $L$-group can be written as 
\[ {}^L\G = \GSO(2n,\C) \rtimes W_{k}, \] 
where the Weil group acts through the quotient 
\[ W_{k} / W_{K} \cong \operatorname{Gal}(K/k). \] 
The $L$-group of $\GL(2n)$ is $\GL(2n,\C) \times W_{k}$, a direct product because 
$\GL(2n)$ is split. These are the Weil forms of the $L$-group, or we can equivalently 
use the Galois forms of the $L$-groups. 

We define a map 
\begin{eqnarray}\label{iota} 
\iota : \GSO(2n,\C) \rtimes \Gamma_{k} \longrightarrow \GL(2n,\C) \times \Gamma_{k} \\
(g,\gamma) \mapsto \begin{cases}
(g,\gamma) & \mbox{if } \gamma_{\vert K} = 1, \\
(h g h^{-1},\gamma) & \mbox{if } \gamma_{\vert K} \not= 1,  
\end{cases}
\end{eqnarray} 
where $\gamma\in\Gamma_{k}$, $g\in\GSO(2n,\C)\subset\GL(2n,\C)$, and 
$h=h^{-1}$ is the matrix 
\[ \left(\begin{matrix}
	I_{n-1} &&& \\
	 & 0 & 1 & \\
	& 1 & 0 & \\
	& & & I_{n-1} \end{matrix}\right). \] 
(We refer to \cite[\S 7.1]{cpss-clay} for more details.) 
The map $\iota$ is an $L$-homomorphism. We also have a compatible family of local 
$L$-homomorphisms $\iota_{v} : {}^{L}\G_{v} \longrightarrow \GL(2n,\C) \times W_v.$ 
Our purpose in this section is to prove the existence of a weak transfer of globally 
generic, cuspidal, automorphic representations of $G=\G(\A)$ to automorphic 
representations of $\GL(2n,\A)$ associated with $\iota$.  

\begin{thm}\label{q-split-weak}  
Let $K/k$ be a quadratic extension of number fields and let $\G=\GSpin^*(2n)$ be as above.  
Let $\psi$ be a nontrivial continuous additive character of $k \backslash \A_k$. The choice 
of $\psi$ and the splitting above defines a non-degenerate 
additive character of $\U(k) \backslash \U(\A)$, again denoted by $\psi$. 

Let $\pi=\otimes_v \pi_v$ be an irreducible, globally $\psi$-generic, cuspidal, automorphic 
representation of $G=\G(\A_k)$. Write $\psi=\otimes_v \psi_v$. 
Let $S$ be a nonempty finite set of non-archimedean 
places $v$ of $k$ such that for every non-archimedean $v \not \in S$ 
both $\pi_v$ and $\psi_v$, as well as $K_{w}/k_{v}$ for $w \vert v$, are unramified. 
Then there exists an automorphic representation $\Pi=\otimes_v \Pi_v$ of $\GL(2n,\A_k)$ 
such that for all $v\not\in S$ the homomorphism parametrizing the local representation $\Pi_v$ is 
given by 
\[ \Phi_v = \iota_{v} \circ \phi_v : W_{k_v} \to \GL(2n,\C), \]
where $W_{k_v}$ denotes the local Weil group of $k_v$ and $\phi_v : W_{k_v} \longrightarrow {}^L\G$ 
is the homomorphism parametrizing $\pi_v$. Moreover, if $\omega_\Pi$ and $\omega_\pi$  
denote the central characters of $\Pi$ and $\pi$, respectively, then $\omega_\Pi = \omega_\pi^n \mu$, 
where $\mu$ is a nontrivial quadratic id\`ele class character. 
Furthermore, $\Pi$ and $\w{\Pi}\otimes\omega_\pi$ are nearly equivalent. 
\end{thm}

\begin{rem}
We proved an analogous result for the split groups $\GSpin(2n+1)$ and $\GSpin(2n)$ in 
\cite{duke}.
\end{rem}

To prove the theorem we will use a suitable version of the converse theorems of Cogdell 
and Piatetski-Shapiro \cite{cps1,cps2}. 
The exact version we need can be found in \cite[\S 2]{cpss-clay} which we quickly review 
below. Next we introduce an irreducible, admissible 
representation $\Pi$ of $\GL(2n,\A)$ as a candidate for the transfer of $\pi$. We then 
prove that $\Pi$ satisfies the required conditions of the converse theorem and hence is  
automorphic. Along the way we also verify the remaining properties of $\Pi$ stated 
in Theorem \ref{q-split-weak}.

\subsection{The Converse Theorem}

Let $k$ be a number field and fix a non-empty finite set $S$ of non-archimedean places of $k$. 
For each integer $m$ let 
\[ \mathcal A_{0}(m) = \left\{\tau | \tau \text{ is a cuspidal representation of } \GL(m,\A_{k}) \right\} \]
and 
\[ \mathcal A^{S}_{0}(m) = \left\{ \tau \in \mathcal A_{0}(m) | \tau_{v} \text{ is unramified 
for all} v \in S\right\}. \] 
Also, for a positive integer $N$ let 
\[ \mathcal T(N-1) = \coprod_{m=1}^{N-1} \mathcal A_{0}(m) \quad \text{and} \quad 
\mathcal T^{S}(N-1) = \coprod_{m=1}^{N-1} \mathcal A^{S}_{0}(m) \] 
and for $\eta$ a continuous character of $k^{\times}\backslash \A^{\times}_{k}$ let 
\[ \mathcal T(S;\eta) = \mathcal T^{S}(N-1)\otimes \eta = 
\left\{ \tau = \tau'\otimes\eta | \tau' \in\mathcal T^{S}(N-1)\right\}. \]
For our purposes we will apply the following theorem with $N=2n$. 

\begin{thm}\label{converse}(Converse theorem of Cogdell and Piatetski-Shapiro) 
Let $\Pi = \otimes \Pi_{v}$ be an irreducible, admissible representation of 
$\GL(N,\A_{k})$ whose central character $\omega_{\Pi}$ is invariant under $k^{\times}$ 
and whose $L$-function 
\[ L(s,\Pi) = \prod_{v} L(s,\Pi_{v}) \] 
is absolutely convergent in 
some right half plane. Let $S$ be a finite set of non-archimedean places of $k$ and let 
$\eta$ be a continuous character of $k^{\times}\backslash \A^{\times}$. Suppose that 
for every $\tau \in \mathcal T(S;\eta)$ the $L$-function $L(s,\tau\times\Pi)$ is {\it nice}, 
i.e., it satisfies the following three conditions: 
\begin{itemize}
\item[(1)] $L(s,\tau\times\Pi)$ and $L(s,\w\tau\times\w\Pi)$ extend to entire functions of $s \in \C$.  
\item[(2)] $L(s,\tau\times\Pi)$ and $L(s,\w\tau\times\w\Pi)$ are bounded in vertical strips. 
\item[(3)] The functional equation $L(s,\tau\times\Pi)=\epsilon(s,\tau\times\Pi) L(s,\w\tau\times\w\Pi)$ 
holds.
\end{itemize} 
Then there exists an automorphic representation $\Pi'$ of $\GL(N,\A_{k})$ such that 
$\Pi_{v} \cong \Pi'_{v}$ for all $v \not\in S$. 
\hfill$\Box$
\end{thm}

The twisted $L$- and $\epsilon$-factors in the statement are those in \cite{cps1}. In particular, 
they are Artin factors and known to be the same as the ones coming from the Langlands-Shahidi 
method at all places. 

\subsection{$L$-functions for $\GL(m) \times \GSpin^{*}(2n)$}

Let $\pi$ be an irreducible, admissible, globally generic representation of $\GSpin^{*}(2n,\A_{k})$ 
and let $\tau$ be a cuspidal representation of $\GL(m, \A_{k})$ with $m \ge 1$. The group 
$\GSpin^{*}(2(m+n))$ has a standard maximal Levi $\GL(m) \times \GSpin^{*}(2n)$ and we have 
the completed $L$-functions 
\[ L(s,\tau \times \pi) = \prod_{v} L(s,\tau_{v}\times\pi_{v}) 
= \prod_{v} L(s,\tau_{v}\otimes\pi_{v}, \iota'_{v}\otimes\iota_v) = L(s, \tau\otimes\pi, \iota'\otimes\iota), \] 
with similar $\epsilon$- and $\gamma$-factors, 
defined via the Langlands-Shahidi method in \cite{shahidi1990-annals}.
Here, $\iota$ is the representation of the $L$-group of $\GSpin^{*}(2n)$ we described before 
and $\iota'$ is the projection map onto the first factor in the $L$-group 
${}^{L}\GL(m) = \GL(m,\C) \times W_{k}$. 

\begin{prop}\label{nice}
Let $S$ be a non-empty finite set of finite places of $k$ and let $\eta$ be a character of 
$k^{\times}\backslash \A^{\times}_{k}$ such that, for some $v \in S$, $\eta_{v}^{2}$ is ramified. 
Then for all $\tau \in \mathcal T(S;\eta)$ the $L$-function $L(s,\tau \times \pi)$ is nice, 
i.e., it satisfies the following three conditions: 
\begin{itemize}
\item[(1)] $L(s,\tau\times\pi)$ and $L(s,\w\tau\times\w\pi)$ extend to entire functions of $s \in \C$.  
\item[(2)] $L(s,\tau\times\pi)$ and $L(s,\w\tau\times\w\pi)$ are bounded in vertical strips. 
\item[(3)] The functional equation $L(s,\tau\times\pi)=\epsilon(s,\tau\times\pi) L(s,\w\tau\times\w\pi)$ 
holds.
\end{itemize} 
\end{prop}

\begin{proof}
Twisting by $\eta$ is necessary for conditions (1) and (2). Both (2) and (3) hold in wide generality. 

Condition (2) follows from \cite[Cor. 4.5]{gelbart-shahidi} and is valid for all $\tau \in \mathcal T(N-1)$, provided 
that one removes neighborhoods of the finite number of possible poles of the $L$-function. 
Condition (3) is a consequence of \cite[Thm. 7.7]{shahidi1990-annals} and is valid for 
all $\tau\in\mathcal T(N-1)$. 

Condition (1) follows from a more general result, \cite[Prop. 2.1]{kim-shahidi-annals}. Note that 
this result rests on Assumption 1.1 of \cite{kim-shahidi-annals}, 
sometimes called Assumption A \cite{kim-israelJ},  
on certain normalized intertwining operators being holomorphic and non-zero. Fortunately the 
assumption has been verified in our cases. The assumption requires two ingredients: the so-called 
``standard modules conjecture'' and the ``tempered $L$-functions conjecture''. Both of these 
have been verified in our cases in Wook Kim's thesis \cite{w-kim}. For results proving various cases 
of this assumption we refer to \cite{shahidi1990-annals, cas-sha, muic-shahidi, muic, asgari, 
kim2005, hei-muic2007, kim-kim-shahidi-vol}. Recently V. Heiermann and E. Opdam have proved 
the assumption in full generality in \cite{hei-opd-ppt}.   
\end{proof}

The key now is to relate the $L$-functions $L(s,\tau \times \pi)$, defined via the Langlands-Shahidi 
method, to the $L$-functions $L(s,\Pi\times\tau)$ in the converse theorem. We note that, when we 
introduce our candidate for $\Pi$, for archimedean places and those non-archimedean places at 
which all data are unramified we know the equality of the local $L$-functions. However, we do not 
know this to be the case for ramified places. We get around this problem through stability of 
$\gamma$-factors, which basically makes the choice of local components of $\Pi$ at the ramified 
places irrelevant as long as we can twist by highly ramified characters. 

\subsection{Stability of $\gamma$-factors}
In this subsection let $F$ denote a non-archimedean local field of characteristic zero. 
Let $G=\GSpin^{*}(2n,F)$, where the quasi-split non-split group is associates with a 
quadratic extension $E/F$. 

Fix a nontrivial additive character $\psi$ of $F$. Let $\pi$ be an irreducible, admissible, 
$\psi$-generic representation of $G$ and let $\eta$ denote a continuous character of $\GL(1,F)$. 
Let $\gamma(s,\eta\times\pi,\psi)$ be the associated $\gamma$-factor defined via the 
Langlands-Shahidi method \cite[Theorem 3.5]{shahidi1990-annals}. We have 
\[ \gamma(s,\eta\times\pi,\psi) = 
\frac{\epsilon(s,\eta\times\pi,\psi)L(1-s,\eta^{-1}\times\w\pi)}{L(s,\eta\times\pi)}. \] 

\begin{prop}\label{stab}
Let $\pi_{1}$ and $\pi_{2}$ be two irreducible, admissible, $\psi$-generic representations of $G$ 
having the same central characters. Then for a suitably highly ramified character $\eta$ of $\GL(1,F)$ 
we have 
\[ \gamma(s,\eta\times\pi_{1},\psi_{v}) = \gamma(s,\eta\times\pi_{2},\psi_{v}). \]  
\end{prop}

\begin{proof}
This is a special case of a more general theorem which is the main result of \cite{cpss-stab}. 
We note that in our case one has to apply that theorem to the self-associate maximal Levi 
subgroup $\GL(1)\times\GSpin^{*}(2n)$ in $\GSpin^{*}(2n+2)$ which does satisfy the assumptions 
of that theorem. 
\end{proof}

\subsection{The Candidate Transfer}
We construct a candidate global transfer $\Pi = \otimes_{v} \Pi_{v}$ as a restricted tensor product of its 
local components $\Pi_{v}$, irreducible, admissible representations of $\GL(2n,k_{v})$. 
There are three cases to consider: (i) archimedean $v$, (ii) non-archimedean unramified $v$, 
(iii) non-archimedean ramified $v$. 

\subsubsection{The archimedean transfer}\label{candidate-inf} 
If $v$ is an archimedean place of $k$, then by 
the local Langlands correspondence \cite{langlands-real, borel-corvallis} the representation 
$\pi_{v}$ is parametrized by an admissible homomorphism $\phi_{v}$ and we choose $\Pi_{v}$ 
to be the irreducible, admissible representation of $\GL(2n,k_{v})$ parametrized by $\Phi_{v}$ as 
in the statement of Theorem \ref{q-split-weak}. We then have 
\begin{equation} \label{L-inf}
L(s,\pi_{v}) = L(s,\iota_{v}\circ\phi_{v}) = L(s,\Pi_{v}) 
\end{equation}
and 
\begin{equation} \label{e-inf}
\epsilon(s,\pi_{v},\psi_{v}) = \epsilon(s,\iota_{v}\circ\phi_{v},\psi_{v}) 
= \epsilon(s,\Pi_{v},\psi_{v}), 
\end{equation} 
where the middle factors are the local Artin-Weil $L$- and $\epsilon$-factors attached to 
representations of the Weil group as in \cite{tate}. The other $L$- and $\epsilon$-factors 
are defined via the Langlands-Shahidi method which, in the archimedean case, are known to 
be the same as the Artin factors defined through the arithmetic Langlands classification 
\cite{shahidi:duke85}. 

If $\tau_{v}$ is an irreducible, admissible representation of $\GL(m,k_{v})$, then it is 
parametrized by an admissible homomorphism $\phi'_{v}:W_{k_{v}}\longrightarrow \GL(m,\C)$ 
and the tensor product homomorphism 
$(\iota_{v} \circ \phi_{v}) \otimes (\iota'_{v}\circ\phi'_{v}) : W_{k_{v}} \longrightarrow \GL(2mn,\C)$ 
is another admissible homomorphism and we again have 
\begin{equation} \label{L-inf-tau} 
L(s,\pi_{v}\times \tau_{v}) = L(s,(\iota_{v} \circ \phi_{v}) \otimes (\iota'_{v}\circ\phi'_{v})) 
= L(s,\Pi_{v} \times \tau_{v}) 
\end{equation}
and 
\begin{equation} \label{e-inf-tau} 
\epsilon(s,\pi_{v} \times \tau_{v},\psi_{v}) = 
\epsilon(s,(\iota_{v} \circ \phi_{v}) \otimes (\iota'_{v}\circ\phi'_{v}),\psi_{v}) 
= \epsilon(s,\Pi_{v} \times \tau_{v},\psi_{v}). 
\end{equation}
Here, $\iota_{v}'$ is just the identity map on $\GL(m,\C)$. 
Hence, we get the following matching of the twisted local $L$- and $\epsilon$-factors. 

\begin{prop}\label{equality-inf}
Let $v$ be an archimedean place of $k$ and let $\pi_{v}$ be an irreducible, admissible, 
generic representation of $\GSpin^{*}(2n,k_{v})$, $\Pi_{v}$ its local functorial transfer 
to $\GL(2n,k_{v})$, and $\tau_{v}$ an irreducible, admissible, generic representation 
of $\GL(m,k_{v})$. Then 
\[ L(s,\pi_{v} \times \tau_{v}) = L(s, \Pi_{v} \times \tau_{v}), \quad  
L(s,\w{\pi}_{v} \times \w{\tau}_{v}) = L(s, \w{\Pi}_{v} \times \w{\tau}_{v}) \] 
and 
\[ \epsilon(s,\pi_{v} \times \tau_{v}, \psi_{v}) = \epsilon(s,\Pi_{v} \times \tau_{v}, \psi_{v}). \]
\hfill$\Box$
\end{prop}

\subsubsection{The non-archimedean unramified transfer}
If $v$ is a non-archimedean place of $k$ such that $\pi_{v}$ as well as all 
$K_{w}/k_{v}$, for $w \vert v$, are unramified, then by the arithmetic Langlands 
classification or the Satake classification \cite{borel-corvallis, satake}, 
the representation $\pi_{v}$ is 
parametrized by an unramified admissible homomorphism 
$\phi_{v}: W_{k_{v}} \longrightarrow {}^{L}\G_{v}^{0}$. 
Again we take $\Phi_{v}$ as in the statement of the theorem. It defines an irreducible, 
admissible, unramified representation $\Pi_{v}$ of $\GL(2n, k_{v})$ \cite{harris-taylor,henniart}. 

Given that $\pi_{v}$ is unramified its parameter $\phi_{v}$ factors through an unramified 
homomorphism into the maximal torus ${}^{L}\T_{v} \hookrightarrow {}^{L}\G_{v}$. Then 
$\Phi_{v}$ has its image in a torus of $\GL(2n,\C)$ which is split and $\Pi_{v}$ is the 
corresponding unramified representation. Then we have 

\begin{equation} \label{L-unram} 
L(s,\pi_{v}) =  L(s,\Pi_{v}) 
\end{equation} 
and 
\begin{equation} \label{e-unram} 
\epsilon(s,\pi_{v},\psi_{v}) = \epsilon(s,\Pi_{v},\psi_{v}) 
\end{equation} 
and the factors on either side of the above equations can be expressed as products of 
one-dimensional abelian Artin factors by multiplicativity of the local factors. 

Let $\tau_{v}$ be an irreducible, admissible, generic, unramified representation of $\GL(m,k_{v})$. 
Again appealing to the general multiplicativity of local factors \cite{jpss-factor,shahidi1990-annals,shahidi-ps-vol} 
we have 
\begin{equation} \label{L-unram-tau} 
L(s,\pi_{v}\times \tau_{v}) = L(s,\Pi_{v} \times \tau_{v}) 
\end{equation} 
and 
\begin{equation} \label{e-unram-tau} 
\epsilon(s,\pi_{v} \times \tau_{v},\psi_{v}) = \epsilon(s,\Pi_{v} \times \tau_{v},\psi_{v}). 
\end{equation} 
Hence, we again get the following matching of the twisted local $L$- and $\epsilon$-factors. 

\begin{prop}\label{equality-unram}
Let $v$ be a non-archimedean place of $k$ and let $\pi_{v}$ be an irreducible, admissible, 
generic, unramified representation of $\GSpin^{*}(2n,k_{v})$, $\Pi_{v}$ its local functorial transfer 
to $\GL(2n,k_{v})$, and $\tau_{v}$ an irreducible, admissible, generic representation 
of $\GL(m,k_{v})$. Then 
\[ L(s,\pi_{v} \times \tau_{v}) = L(s, \Pi_{v} \times \tau_{v}), \quad  
L(s,\w{\pi}_{v} \times \w{\tau}_{v}) = L(s, \w{\Pi}_{v} \times \w{\tau}_{v}) \] 
and 
\[ \epsilon(s,\pi_{v} \times \tau_{v}, \psi_{v}) = \epsilon(s,\Pi_{v} \times \tau_{v}, \psi_{v}). \]
\hfill$\Box$
\end{prop}

We also make the local transfer in this case explicit.  The analysis is similar to that of 
the quasi-split $\SO(2n)$ carried out in \cite[\S 7.2]{cpss-clay}, which we refer to for more detail. 

The unramified principal series representation $\pi_{v}$ is given by an unramified character 
\[ \chi = (\chi_{0},\chi_{1},\dots,\chi_{n-1},\chi^{1}_{n}) \] 
of the maximal torus in $\GSpin^{*}(2n,k_{v})$. Here, $\chi_{0}, \chi_{1},\dots, \chi_{n-1}$ are 
characters of $k_{v}^{\times}$ with $\chi_{0}$ being the central character of $\pi_{v}$ 
and $\chi^{1}_{n}$ is a character of $K_{w}^{1}$, elements of 
norm one in $K_{w}$ embedded in $\GL(2,k_{v})$ as in \cite{langlands-labesse}. By Hilbert 
Theorem 90 we have $K_{w}^{\times}/k_{v}^{\times} \cong K_{w}^{1}$, which allows us to 
extend $\chi^{1}_{n}$ to a character $\tilde{\chi}_{n}$ of $K_{w}^{\times}$. The local transferred 
representation $\Pi_{v}$ is then given by the character 
\[ \tilde{\chi} = \left(\chi_{1},\dots,\chi_{n-1},\tilde{\chi}_{n},\chi_{n-1}^{-1}\chi_{0}, 
\dots,\chi_{1}^{-1}\chi_{0} \right) \] 
of a (non-split) torus in $\GL(2n,k_{v})$. (See \cite[\S 6]{duke} for the split case.) 
To get principal series on $\GL(2n,k_{v})$ the 
character $\tilde{\chi}_{n}$ must factor through the norm map. Write 
\[ \tilde{\chi}_{n} = \chi_{n} \circ N_{K_{w}/k_{v}} \] 
with $\chi_{n}$ a character of $k_{v}^{\times}$ satisfying $\chi_{n}^{2} = \chi_{0}$. 
We get the principal series representation 
\[ I(\chi_{n}\mu_{K_{w}/k_{v}},\chi_{n}) = I(\chi_{n}\mu_{K_{w}/k_{v}}, \chi_{n}^{-1}\chi_{0}) \] 
of $\GL(2,k_{v})$, where $\mu_{K_{w}/k_{v}}$ is the quadratic character of $k_{v}^{\times}$ 
associated with the quadratic extension $K_{w}/k_{v}$ by local class field theory. Hence, $\pi_{v}$ 
transfers to the unramified principal series representation of $\GL(2n,k_{v})$ induced from the character 
\[ \left(\chi_{1},\dots\chi_{n-1}, \chi_{n}\mu_{K_{w}/k_{v}}, \chi_{n}^{-1}\chi_{0}, 
\chi_{n-1}^{-1}\chi_{0}, \dots,\chi_{1}^{-1}\chi_{0}\right). \] 
It is now clear that the central character of $\Pi_{v}$ is $\chi_{0}^{n} \mu_{K_{w}/k_{v}}$ 
and its contragredient is isomorphic to $\Pi_{v} \otimes \chi_{0}^{-1}$. Therefore, we have 
proved the following. 

\begin{prop}\label{omega-tilde}
Let $v$ be a non-archimedean place of $k$ and let $\pi_{v}$ be an irreducible, admissible, 
generic, unramified representation of $\GSpin^{*}(2n,k_{v})$ with $\Pi_{v}$ its local 
functorial transfer to $\GL(2n,k_{v})$ defined above. Then 
\[ \omega_{\Pi_{v}} = \omega_{\pi_{v}}^{n} \mu_{v} \] 
and 
\[ \Pi_{v} \cong \tilde{\Pi}_{v} \otimes \omega_{\pi_{v}}. \] 
Here, $\mu_{v}$ is a quadratic character of $k_{v}^{\times}$ associated with $\GSpin^{*}(2n,k_{v})$. 
\hfill$\Box$
\end{prop}

\subsubsection{The non-archimedean ramified transfer}\label{candidate-ram}
For $v$ a non-archimedean ramified place of $k$ we take $\Pi_{v}$ to be an arbitrary, irreducible, 
admissible representation of $\GL(2n,k_{v})$ whose central character satisfies 
\[ \omega_{\Pi_{v}} = \omega_{\pi_{v}}^{n} \mu_{v}, \] 
where $\mu_{v}$ is the quadratic character associated with the quadratic extension $K_{w}/k_{v}$.  

We can no longer expect equality of $L$- and $\epsilon$-factors as in the previous cases. However, 
we do still get equality if we include a highly ramified character thanks to stability of $\gamma$-factors.

\begin{prop}\label{equality-ram}
Let $v$ be a non-archimedean ramified place of $k$ and let $\pi_{v}$ be an irreducible, admissible, 
generic representation of $\GSpin^{*}(2n,k_{v})$ and let $\Pi_{v}$ be an irreducible, admissible 
representation of $\GL(2n,k_{v})$ as above. If $\tau_{v}=\tau'_{v}\otimes\eta_{v}$ is an irreducible, admissible, generic 
representation of $\GL(m,k_{v})$ with $\eta_{v}$ a sufficiently ramified character of $\GL(1,k_{v})$, 
then 
\[ L(s,\pi_{v} \times \tau_{v}) = L(s, \Pi_{v} \times \tau_{v}), \quad  
L(s,\w{\pi}_{v} \times \w{\tau}_{v}) = L(s, \w{\Pi}_{v} \times \w{\tau}_{v}) \] 
and 
\[ \epsilon(s,\pi_{v} \times \tau_{v}, \psi_{v}) = \epsilon(s,\Pi_{v} \times \tau_{v}, \psi_{v}). \]
\end{prop}

\begin{proof}
The representation $\tau'_{v}$ can be written as a full induced principal series  
\[ \tau_{v} = \operatorname{Ind}\left(\nu^{b_{1}}\otimes\cdots\otimes\nu^{b_{m}} \right)\otimes\eta_{v} 
= \operatorname{Ind}\left(\eta_{v}\nu^{b_{1}}\otimes\cdots\otimes\eta_{v}\nu^{b_{m}} \right), \] 
where $\nu(\cdot)=|\cdot|_{v}$. By multiplicativity of the $L$- and $\epsilon$-factors we have  
\[ L(s,\pi_{v} \times \tau_{v}) = \prod_{i=1}^{m} L(s+b_{i},\pi_{v}\times\eta_{v}) \] 
and 
\[ \epsilon(s,\pi_{v} \times \tau_{v},\psi_{v}) = \prod_{i=1}^{m} L(s+b_{i},\pi_{v}\times\eta_{v},\psi_{v}).  \] 
Similarly, 
\[ L(s,\Pi_{v} \times \tau_{v}) = \prod_{i=1}^{m} L(s+b_{i},\Pi_{v}\times\eta_{v}) \] 
and 
\[ \epsilon(s,\Pi_{v} \times \tau_{v},\psi_{v}) = \prod_{i=1}^{m} L(s+b_{i},\Pi_{v}\times\eta_{v},\psi_{v}).  \] 
This reduces the proof to the case of $m=1$. 

Next, note that because $\eta_{v}$ is sufficiently ramified (depending on $\pi_{v}$) the $L$-functions 
stabilize to one and we have 
\[ L(s,\pi_{v}\times\eta_{v}) \equiv 1 \] 
and 
\[ \epsilon(s,\pi_{v}\times\eta_{v},\psi_{v}) = \gamma(s,\pi_{v}\times\eta_{v},\psi_{v}). \] 
On the other hand, by stability of gamma factors Proposition \ref{stab} we may replace $\pi_{v}$ 
with another representation with the same central character. Hence, for $n$ arbitrary characters 
$\chi_{1},\chi_{2},\dots,\chi_{n}$, $\chi_{0}=\omega_{\pi_{v}}$ and the quadratic character 
$\mu_{v}$ we have 
\begin{eqnarray*} 
\gamma(s,\pi_{v}\times\eta_{v},\psi_{v}) &=& 
\left(\prod_{i=1}^{n-1} \gamma(s,\eta_{v}\chi_{i},\psi_{v}) \gamma(s,\eta_{v}\chi_{i}^{-1}\chi_{0},\psi_{v}) \right) \\
& \cdot & \gamma(s,\eta_{v}\chi_{n}\mu,\psi_{v}) \gamma(s,\eta_{v}\chi_{n}^{-1}\chi_{0},\psi_{v}). 
\end{eqnarray*}
We refer to \cite[\S 6]{duke} for more details in the split case. The calculations in the quasi-split case 
are similar, the only difference being the appearance of the quadratic character $\mu_{v}$. 

We have similar relations also for the $\GL$ case. More precisely, because 
$\omega_{\Pi_{v}} = \chi_{0}^{n}\mu$, by \cite[Proposition 2.2]{js} we have 
\[ L(s,\Pi_{v}\times\eta_{v}) \equiv 1 \] 
and 
\begin{eqnarray*} 
\epsilon(s,\Pi_{v} \times \eta_{v},\psi_{v}) &=& 
\left(\prod_{i=1}^{n-1} \gamma(s,\eta_{v}\chi_{i},\psi_{v}) \gamma(s,\eta_{v}\chi_{i}^{-1}\chi_{0},\psi_{v}) \right) \\
&\cdot& \gamma(s,\eta_{v}\chi_{n}\mu,\psi_{v}) \gamma(s,\eta_{v}\chi_{n}^{-1}\chi_{0},\psi_{v}). 
\end{eqnarray*}
Note that this is a special case of the multiplicativity of the local factors. This gives the equalities 
for the case of $m=1$ and hence completes the proof. 
\end{proof}

\subsection{Proof of Theorem \ref{q-split-weak}}

Let $\omega$ denote the central character $\omega_\pi$ of $\pi$ and let $S$ be as in 
the statement of Theorem \ref{q-split-weak}. 
We let $\Pi = \otimes_v \Pi_v$ with $\Pi_{v}$ the candidates we constructed 
in \ref{candidate-inf}--\ref{candidate-ram}. Also, let $\mu = \otimes_{v} \mu_{v}$ 
be a quadratic id\`ele class character associated, by class field theory, with the quadratic 
extension $K/k$, where $K$ is the field over which $\GSpin^{*}(2n)$ is split. 

Choose an id\`ele class character $\eta$ of $k$ which is sufficiently ramified at places $v \in S$ 
so that the requirements of Propositions \ref{nice} and \ref{stab} are satisfied. We apply 
Theorem \ref{converse} to the representation $\Pi$ and $\mathcal T(S;\eta)$ with $S$ and 
$\eta$ as above. 

By construction the central character $\omega_{\Pi}$ of $\Pi$ is equal to $\omega^{n} \mu$. 
Therefore, it is invariant under $k^{\times}$. Moreover, by (\ref{L-inf}) and (\ref{L-unram}) we have 
\[ L^{S}(s,\Pi) = \prod_{v\not\in S} L(s,\Pi_{v}) = \prod_{v\not\in S} L(s,\pi_{v}) = L^{S}(s,\pi). \] 
This implies that $L(s,\Pi) = \prod_{v} L(s,\Pi_v)$ is absolutely convergent in some right half plane. 
Furthermore, by Propositions \ref{equality-inf}, \ref{equality-unram}, and \ref{equality-ram} we 
are free to check the remaining properties of being nice for the $L$-functions $L(s,\Pi\times\tau)$ 
and the corresponding $\epsilon$-factors, for $\tau \in \mathcal T(S;\eta)$, instead for 
the $L$-functions $L(s,\pi\times\tau)$ and its corresponding $\epsilon$-factors. 
Now the converse theorem can be applied thanks to Proposition \ref{nice} to conclude that 
there exists an automorphic representation of $\GL(2n,\A_{k})$ whose local components 
at $v \not\in S$ agree with those of $\Pi$. This automorphic representation is what we are 
calling $\Pi$ in the statement of the theorem. 

Finally, by Proposition \ref{omega-tilde}, outside the finite set $S\cup\left\{v : v|\infty \right\}$, 
the central character of the automorphic representation $\Pi$ agrees with the id\`ele class character 
$\omega^{n}\mu$, which implies that it is equal to $\omega^{n}\mu$. The same proposition 
also gives that $\Pi_{v} \cong \w{\Pi}_{v} \otimes \omega_{\pi_{v}}$ for 
$v \not\in S\cup\left\{v : v|\infty \right\}$. This completes the proof.  
\hfill$\Box$

\section{The integral for $\GL(m)\times\G(n)$} \label{grs}

In this section we develop the analogue of the theory of Gelbart, Ginzburg, Piatetski-Shapiro, 
Rallis and Soudry in our situation. 

Let $n \ge 0$ and denote by $\G=\G(n)$ either $\GSpin(2n+1)$ or a quasi-split 
$\GSpin(2n)$. We write down and analyze a zeta integral which will give the $L$-functions 
for $\GL(m) \times \G(n)$ for $m \le n$. For $m=n$ the construction and analysis of 
this zeta integral when $\G$ is a special orthogonal group is due to Gelbart and Piatetski-Shapiro 
\cite{gps} using their Methods A and B. (They also cover $\G$ symplectic.) 
Ginzburg then extended their work to the case of $m < n$ and $\G$ a special orthogonal 
group in \cite{ginzburg}. 
We would like to carry out the same construction and analysis for $m \le n$ and $\G$ a 
general spin group. We first need to define a certain unipotent subgroup of $\G$ just as in 
\cite{ginzburg}. For $m < n$ we use the embeddings we described in Section \ref{embed}. 
We point out that this subgroup will be trivial if $m=n$ and Ginzburg's integral reduces 
to that of \cite{gps}.

\subsection{Unipotent subgroups} 
We define a subgroup $\N$ of $\G(n)$ generated by the subgroups 
$\X$, $\Y$, and $\U_{2(n-m)+1}$ in the odd case or $\U_{2(n-m)}$ in the even case, 
as follows. 

Each of the three subgroups is generated by a family of root groups. Let 
\[ \U_{2(n-m)+1} = \left\langle \U_{\alpha} \left| 
\alpha = \begin{cases} e_{i} \pm e_{j}, & m+1 \le i < j \le n \\
e_{\ell}, &  m+1 \le \ell \le n \end{cases} \right.
\right\rangle \] 
in the odd case or 
\[ \U_{2(n-m)} = \left\langle \U_{\alpha} \left| 
\alpha =  e_{i} \pm e_{j}, \quad m+1 \le i < j \le n \right.
\right\rangle \] 
in the even case, 
i.e., $\alpha$ is a positive root which can be written as a linear combination of 
the simple roots of $\G(n)$ not involving the first $m$. Clearly $\U_{2(n-m)+1}$ 
embeds naturally in the maximal unipotent subgroup $\U=\U_{2n+1}$ of $\GSpin(2n+1)$ 
and $\U_{2(n-m)}$ embeds in $\U=\U_{2n}$, the maximal unipotent subgroup of $\GSpin(2n)$.  

Moreover, let 
\[ \X = \left\langle \U_{-(e_{i}-e_{j})} \left| \begin{array}{c}1 \le i \le m, \\ m+1\le j \le n \end{array}\right.
\right\rangle, \] 
and 
\[ \Y = \left\langle \U_{(e_{i}+e_{j})} \left| \begin{array}{c}1 \le i \le m, \\ m+1\le j \le n \end{array}\right.
\right\rangle \] 
in either case. 

\subsection{The Zeta Integral} 
Let $\pi$ be a generic, cuspidal, automorphic representation of $\G(n,\A)$ and 
let $\tau$ be a cuspidal, automorphic representation of $\GL(m,\A)$. Denote by 
$U=\U(\A)$ the maximal unipotent subgroup of $\G(\A)$ generated by $U_{\alpha}$ for 
$\alpha \in R^{+}$. Every $u \in U$ can be written uniquely as 
\begin{equation} 
u = \prod_{\alpha \in R^{+}} u_{\alpha}(x_{\alpha}).  
\end{equation}

Let $\psi$ be a non-degenerate (additive) character of $k \backslash \A$ and extend 
it to a non-degenerate character, again denoted by $\psi$, of $U$ via 
\begin{equation}
\psi(u) = \psi \left( \sum_{\alpha \in \Delta} x_{\alpha} \right). 
\end{equation} 
We assume that $\pi$ is $\psi$-generic, i.e., the space 
$\mathcal W(\pi,\psi)$ (Whittaker model) of functions 
\begin{equation}
W_{\phi}(g) = \int_{\U(k)\backslash \U(\A)} \phi(ug) \psi(u) du, \quad \phi \in \pi, g \in \G(\A), 
\end{equation} 
is non-zero. Recall that cuspidal, automorphic representations of $\GL(m,\A)$, such as $\tau$, 
are automatically generic. 

The group $\M_{m} = \GL(m) \times \GL(1)$ sits inside either of $\H=\H(m)=\GSpin(2m+1)$ or $\GSpin(2m)$ 
as the Levi component of the standard Siegel parabolic $\P_{m}=\M_{m} \N_{m}$. This parabolic corresponds to the subset 
\[ \theta = \Delta - \{ \alpha_{m}\} \] 
of $\Delta$. We choose $\H$ of the type opposite to $\G$. In other words, if $\G$ is of type $B_{n}$, 
then we take $\H$ of type $D_{m}$ and vice versa.  We also denote the maximal unipotent subgroup 
of $\M_m$ by $\ZZ_m$. Therefore, $\ZZ_m \N_m = \U_\H,$ the maximal unipotent subgroup of $\H=\H(m)$.  

Let $\omega$ be an id\`ele class character of $\GL(1,\A)$.  
Let
\begin{equation} \label{tauprime}
\tau'_{s} = \tau  |\det|^{s- 1/2} \otimes \omega
\end{equation}
be a representation of $\M_{m}(\A)$, and 
extend it trivially across $\N_{m}(\A)$ to obtain a representation of $\P_{m}(\A)$. 
Consider the normalized induced representation  
\begin{equation}\label{induced} 
\operatorname{Ind}_{\P_{m}(\A)}^{\H(\A)} \left(\tau'_{s} \right). 
\end{equation}

For $f_{\tau'_{s}}$ in this induced representation construct the Eisenstein series 
\begin{equation} \label{EisConstruct} 
E(h, f_{\tau'_{s}}) = \sum_{\gamma \in \P_{m}(k) \backslash \H(k)} f_{\tau'_{s}}(\gamma h), 
\quad h \in \H(\A).  
\end{equation}

Moreover, we define the ``quasi-Whittaker functions'' $W_{f_{\tau'_{s}}}$ as follows. 
Recalling that $\ZZ_{m}$ is the maximal unipotent subgroup of the Siegel Levi 
$\M_m=\GL(1) \times \GL(m)$ in $\H$, regard $\psi$ as a character of $\ZZ_{m}(\A)$ 
through $\psi(z) = \psi(\sum z_{i,i+1})$. With $f_{\tau'_{s}}$ as before we define 
\begin{equation}\label{Wf}
W_{f_{\tau'_{s}}}(h) = \int_{\ZZ_{m}(k) \backslash \ZZ_{m}(\A)} f(zh) \psi(z) dz, \quad h\in\H(\A).  
\end{equation}

We are now prepared to define the zeta integral as follows. Let $\phi$ be in the space of 
$\pi$ and let $f_{\tau'_{s}}$ be as above. Recall that $\X$, $\Y$ and $\U_{\ell}$ with $\ell=2(n-m)+1$ 
or $\ell=2(n-m)$ are unipotent subgroups of $\G(n)$ and we have the embedding 
$\i : \H \hookrightarrow \G(n)$, where $\H=\GSpin(2m)$ or $\GSpin(2m+1)$, respectively. 
Define 
\begin{eqnarray} \nonumber
I(\phi,f_{\tau'_{s}}) &=& 
\int\limits_{\X(k) \backslash \X(\A)} \quad 
\int\limits_{\Y(k) \backslash \Y(\A)} \quad 
\int\limits_{\U_{\ell}(k) \backslash \U_{\ell}(\A)} \quad 
\int\limits_{\H(k) \backslash \H(\A)}   \\ 
&& \phi\left[u_{\ell} \cdot y \cdot x \cdot \i(h) \right] 
\cdot \psi(u_{\ell}) 
\cdot E(h,f_{\tau'_{s}}) \, dh \, du_{\ell} \, dy \, dx . \label{zeta}
\end{eqnarray} 

We should remark here that (\ref{zeta}) defines a Rankin-Selberg type integral as follows. Recalling 
that $\N$ denotes 
the subgroup generated by $\X$, $\Y$, and $\U_{\ell}$ and setting  
\[ \phi_{N}(h) = \int\limits_{\N(k) \backslash \N(\A)} \phi\left(n \cdot \i(h)\right) \psi(n) dn, \quad h \in \H(\A), \]
we have 
\[ I(\phi,f_{\tau'_{s}}) = \int\limits_{\H(k) \backslash \H(\A)} \phi_{N}(h) E(h,f_{\tau'_{s}}) dh. \]

\subsection{The Basic Identity} 
We now state and prove the basic identity for the zeta integrals we just introduced. We start 
with a lemma. This lemma is an analogue of Ginzburg's lemma in \cite[p. 172]{ginzburg}. For a special 
more explicit statement of this lemma see \cite[p. 168]{ginzburg}.  The proof of the lemma also 
carries over, more or less word for word, from that of Ginzburg's lemma in \cite[p. 172]{ginzburg}. 
We remark that we are using slightly different notation from Ginzburg's paper. For example, we 
are making a distinction between $\H$ and the image of its embedding in $\G$ through the use of 
the map $\i$. Also we are using the notation $\N_m$ for the unipotent radical of the parabolic 
$\P_m$ reserving $\M_m$ for its Levi component while Ginzburg's paper uses $M_k$ (with his 
$k$ being our $m$) for the unipotent radical. 

\begin{lem}\label{ginzburg-lemma}
With notation as in the previous section we have 
\begin{eqnarray*} 
\int\limits_{\X(\A)} 
\sum\limits_{\gamma \in \ZZ_m(k) \backslash \M_m(k) } W_\phi(\i(\gamma) x g) \, dx 
&=&  
\int\limits_{\N_m(k) \backslash \N_m(\A)} \quad 
\int\limits_{\X(k) \backslash \X(\A)} \quad 
\int\limits_{\Y(k) \backslash \Y(\A)} \quad 
\int\limits_{\U_{\ell}(k) \backslash \U_{\ell}(\A)} \\
&& \phi\left[ u_\ell \cdot y \cdot x \cdot n \cdot g \right] \psi(u_\ell) \, 
du_\ell \, dy \, dx \, dn, \quad \forall g \in \G(\A). 
\end{eqnarray*}
\end{lem}

We now state the basic identity involving the zeta integrals, the main result of this section. 

\begin{thm} \label{id}
\begin{itemize}
\item[(a)] $I(\phi,f_{\tau'_{s}})$ converges for all $s$. 
\item[(b)] \[ I(\phi,f_{\tau'_{s}}) = 
\int\limits_{\U_{\H}(\A) \backslash \H(\A)} \quad 
\int\limits_{\X(\A)} W_{\phi}(\i(h)\, x) \, W_{f_{\tau'_{s}}}(h) \, dx \, dh . \]
\item[(c)] $I(\phi,f_{\tau'_{s}})$ has a meromorphic continuation and 
satisfies the functional equation.  
\[ I\left(\phi,f_{\tau'_{s}}\right) = I\left(\phi, M(s)f_{\tau'_{1-s}}\right). \]
\end{itemize}
\end{thm} 

\begin{proof} 
The proof is similar to that of \cite[Theorem A]{ginzburg} which we closely follow. 

To prove part (a) consider the integral on the right hand side of (\ref{zeta}). First, we would like to 
replace the integration over $\H(k) \backslash \H(\A)$ by a Siegel set. Recall that a 
Siegel set in $\H(\A)$ is a set of the form 
\[ S_{H} = U^{0}_{H} A_{c} K_{H}, \] 
where $U^{0}_{H}$ is a relatively compact subset of the maximal unipotent 
subgroup $\U_{\H}(\A)$ in $\H(\A)$,  $K_{H}$ is a maximal compact subgroup in $\H(\A)$, 
and $A_{c}$ consists of elements $a$ in the maximal torus of $\H(\A)$ satisfying 
$|\alpha(a)| \ge c$ for all simple roots $\alpha$ of $\H$.  By reduction theory we 
know that $\H(\A) = \H(k) S_{H}$. 

Choose a Siegel set $S_{H} = U^{0}_{H} A_{c_{1}} K_{H}$ in $\H(\A)$ as above so that 
(\ref{zeta}) can be written as 
\begin{eqnarray}\label{zeta2} 
\nonumber
I(\phi,f_{\tau'_{s}}) &=& 
\int\limits_{\X(k) \backslash \X(\A)} \quad 
\int\limits_{\Y(k) \backslash \Y(\A)} \quad 
\int\limits_{\U_{\ell}(k) \backslash \U_{\ell}(\A)} \quad 
\int\limits_{U^{0}_{H}} \quad 
\int\limits_{A_{c_{1}}} \quad 
\int\limits_{K_{H}}   \\ 
&& \phi\left[u_{\ell}\,  y\,  x\,  \i(u')\, \i(a)\, \i(k) \right] 
\cdot \psi(u_{\ell}) 
\cdot E(u' a k,f_{\tau'_{s}}) \, du' \, da \, dk \, du_{\ell} \, dy \, dx .
\end{eqnarray} 

Note that we have $x \, \i(u') = \i(u') \, x$. This follows from the definitions of $\X$ and 
of the embedding $\i$ because the root groups defining $X$ and $\i(u')$ commute, a fact 
that follows from \cite[(21)]{duke}, for example. Hence, we are allowed to change the 
order of $x \, \i(u')$ to $\i(u') \, x$ in the integral. 

Next, we choose a Siegel set $S_{G} = U^{0}_{G} R_{c} K_{G}$ in $\G(\A)$ in a similar 
way. Write 
\begin{equation}\label{in-siegel} 
x \, \i(a) = u'' r k, \quad u'' \in U^{0}_{G}, r \in R_{c}, k \in K_{G}. 
\end{equation} 
Then the integral (\ref{zeta2}) becomes 
\begin{equation}\label{zeta3} 
I(\phi,f_{\tau'_{s}}) = %
\int\limits_{U'} 
\int\limits_{R_{c}} 
\int\limits_{K_{G}} %
\phi\left( u'\, r\, k \right) 
\cdot \psi(u') 
\cdot E(u' r k,f_{\tau'_{s}}) \, du' \, dr \, dk , 
\end{equation} 
where $U'$ is some unipotent set in $\G(\A)$. 

Recall that $\phi$ is a cusp form and hence {\it rapidly decreasing}. This means, in particular, 
that for any $N \in \Z$ there exists a constant $C_{\phi,N}$ such 
that 
\begin{equation}\label{rap-dec} 
\left| \phi( u'\, r\, k) \right| \le C_{\phi,N} \left| \alpha(r)\right|^N, 
\end{equation} 
for all simple roots $\alpha$ of $\G$. Moreover, the Eisenstein series $E(u'\, r\, k,f_{\tau'_{s}})$ 
is {\it slowly increasing}. This means, in particular, that (\ref{rap-dec}) holds, with $\phi$ 
replaced by the Eisenstein series, for {\it some} $N$.  
We will use these facts to bound the integral on the right hand side of (\ref{zeta3}). 
Notice that we do not know $r$ explicitly and, in particular, we do not know that 
$|\alpha_i(r)|\to\infty$ for some $i$. Instead, we have to find some other way to bound the 
right hand side of (\ref{zeta3}).   

Consider the equality (\ref{in-siegel}). Writing  
\[ a = e_0^*(t_0) e_1^*(t_1) \dots e_m^*(t_m) \] 
use Lemma \ref{torus-i} to see that 
\[ \i(a) = \begin{cases}
f_0^*(t_0) f_1^*(t_1) \cdots f_{m-1}^*(t_{m-1}) f_m^*(t_m) 
& \text{ if $\H$ is even and $\G$ is odd}, \\
\\
f_0^*(-t_0) f_1^*(t_1) \cdots f_{m-1}^*(t_{m-1}) f_m^*(-t_m)f_n^*(-1) 
& \text{ if $\H$ is odd and $\G$ is even.}
\end{cases} \] 
Also, write 
\[ r = f_0^*(r_0) f_1^*(t_1) \cdots f_n^*(t_n). \] 
Now apply the character $D=f_1+\cdots+f_n$, which lies in the character lattice, to both 
sides of (\ref{in-siegel}). (This character amounts to the determinant on the $\GL(n)$ part 
of the Siegel Levi.) Notice that the elements on both sides are indeed inside the Siegel Levi in $\G$. 
We conclude that 
\[ \left| t_1 t_2 \cdots t_m \right| = \left| r_1 r_2 \cdots r_n \right|. \] 

By definition of a Siegel set, on $R_c$, we have 
\[ |r_1| \ge c |r_2|, \, \dots, \, |r_{n-1} \ge c |r_n| \] 
in both even and odd cases. This implies that 
\begin{eqnarray*} 
\left| t_1 t_2 \cdots t_m \right| &=& \left| r_1 r_2 \cdots r_n \right| \\ 
&\le&  |r_1| \cdot c^{-1} |r_1| \cdot c^{-2} |r_1| \cdots c^{n-1} |r_1| \\
&=& c^{-n(n-1)/2}  |\alpha_1(r) \cdots \alpha_n(r)|^n.  
\end{eqnarray*}
Since $|t_1 t_2 \cdots t_m| \longrightarrow \infty$ we have that $|\alpha_1(r) \cdots \alpha_n(r)| \longrightarrow \infty$. 
Therefore, we conclude from (\ref{zeta3}) that 
\[ 
\left| I(\phi,f_{\tau'_{s}}) \right| \le 
\int\limits_{R_{c}} 
|\alpha_1(r) \cdots \alpha_n(r)|^{-N} p_s(|r_1 \cdots r_n|)  \, dr, \quad \forall N \in \Z, 
\]
where $p_s(|r_1 \cdots r_n|)$ is a polynomial in $|r_1 \cdots r_n|$. Since 
$|\alpha_1(r) \cdots \alpha_n(r)| \longrightarrow \infty$, the last integral converges for any fixed $s$ 
if we take $N$ to be large enough. This proves part (a). 

In order to prove part (b) start with the definition (\ref{zeta}) of $I\left(\phi,f_{\tau'_{s}}\right)$ 
and unfold the Eisenstein series as in (\ref{EisConstruct}) to get 
\begin{eqnarray*}
I(\phi,f_{\tau'_{s}}) 
&=& 
\int\limits_{\X(k) \backslash \X(\A)} \quad  
\int\limits_{\Y(k) \backslash \Y(\A)} \quad 
\int\limits_{\U_{\ell}(k) \backslash \U_{\ell}(\A)} \quad
\int\limits_{\H(k) \backslash \H(\A)} \\
&& \phi\left[ u_\ell \, y \, x \, n \, \i(h) \right] \cdot \psi(u_\ell) \cdot 
\sum\limits_{\gamma \in \P_m(k)\backslash\H(k)} f_{\tau'_{s}}(\gamma h)  
dh \, du_\ell \, dy \, dx \\
&=& 
\int\limits_{\X(k) \backslash \X(\A)} \quad 
\int\limits_{\Y(k) \backslash \Y(\A)} \quad 
\int\limits_{\U_{\ell}(k) \backslash \U_{\ell}(\A)} \quad 
\int\limits_{\P_m(k) \backslash \H(\A)} \\
&& \phi\left[ u_\ell \, y \, x \, n \, \i(h) \right] \cdot \psi(u_\ell) \cdot f_{\tau'_{s}}(\gamma h) 
dh \, du_\ell \, dy \, dx .  
\end{eqnarray*}
Writing 
\[ \int\limits_{\P_m(k) \backslash \H(\A)} = 
\int\limits_{\M_m(k) \N_m(\A) \backslash \H(\A)}  \quad 
\int\limits_{\N_m(k) \backslash \N_m(\A)} \]
and changing the order of integration we get 
\begin{eqnarray*}
I(\phi,f_{\tau'_{s}}) 
&=& 
\int\limits_{\M_m(k) \N_m(\A) \backslash \H(\A)} \quad 
\int\limits_{\N_m(k) \backslash \N_m(\A)} \quad 
\int\limits_{\X(k) \backslash \X(\A)} \quad 
\int\limits_{\Y(k) \backslash \Y(\A)} \quad 
\int\limits_{\U_{\ell}(k) \backslash \U_{\ell}(\A)} \\ 
&& \phi\left[ u_\ell \, y \, x \, n \, \i(h) \right] \cdot \psi(u_\ell) \cdot f_{\tau'_{s}}(n h) 
du_\ell \, dy \, dx \, dn \, dh .  
\end{eqnarray*} 
Now we have $f_{\tau'_{s}}(n h) = f_{\tau'_{s}}(h).$ 
Applying Lemma \ref{ginzburg-lemma} we get 
\begin{eqnarray*}
I(\phi,f_{\tau'_{s}}) 
&=& 
\int\limits_{\M_m(k) \N_m(\A) \backslash \H(\A)}  \quad
\int\limits_{\X(\A)}  \quad 
\sum\limits_{\gamma\in \ZZ_m(k)\backslash\M_m(k)}
W_\phi\left( \i(\gamma) \cdot x \cdot \i(h) \right)
f_{\tau'_{s}}(h) 
 \, dx \, dh .  
\end{eqnarray*} 
Next, because $\X$ is normalized by $\i(\gamma)$, after making a change of variables if necessary, we 
may change $\i(\gamma)\, x $ to $x\, \i(\gamma)$. Moreover, because $f_{\tau'_{s}}$ is in a space induced 
from cusp forms and $\gamma \in \M_m(k)$ we have $f_{\tau'_{s}}(\gamma\, \i(h)) = f_{\tau'_{s}}(\i(h))$. Hence, 
\begin{eqnarray*}
I(\phi,f_{\tau'_{s}}) 
&=& 
\int\limits_{\ZZ_m(k) \N_m(\A) \backslash \H(\A)} \quad
\int\limits_{\X(\A)} 
W_\phi\left( x \cdot \i(h) \right)
f_{\tau'_{s}}(h) 
 \, dx \, dh .  
\end{eqnarray*} 
Furthermore, we may change the order of integration from $x\, \i(h)$ to $\i(h)\, x$ since 
$x\, \i(h) = u'\, \i(h)\, x'$ with $x'$ in the maximal unipotent subgroup of $\G(\A)$ only involving 
non-simple roots, hence $\psi(u') = 1$. Again writing 
\[ \int\limits_{\ZZ_m(k) \N_m(\A) \backslash \H(\A)} = 
\int\limits_{\ZZ_m(\A) \N_m(\A) \backslash \H(\A)}  \quad 
\int\limits_{\ZZ_m(k) \backslash \ZZ_m(\A)} \]
we get 
\begin{eqnarray*}
I(\phi,f_{\tau'_{s}}) 
&=& 
\int\limits_{\ZZ_m(\A) \N_m(\A) \backslash \H(\A)} \quad 
\int\limits_{\X(\A)} 
W_\phi\left( \i(h) x \right)
\left[
\int\limits_{\ZZ_m(k) \backslash \ZZ_m(\A)} 
\psi(z) )f_{\tau'_{s}}(z h), dz  
\right]
\, dh \, dx  .  
\end{eqnarray*} 
Using (\ref{Wf}) we finally get 
\begin{eqnarray*}
I(\phi,f_{\tau'_{s}}) 
&=& 
\int\limits_{\ZZ_m(\A) \N_m(\A) \backslash \H(\A)} \quad 
\int\limits_{\X(\A)} 
W_\phi\left( \i(h)\, x \right)  W_{f_{\tau'_{s}}}(h) \, dx \, dh .
\end{eqnarray*}
This is part (b). Part (c) follows from the definition (\ref{zeta}) and the 
properties of the meromorphic continuation and the functional equation of 
the Eisenstein series (\ref{EisConstruct}). 
\end{proof}
 
As a consequence of Theorem \ref{id} we have the following Euler product expansion. 
With appropriate choices as in \cite[pp. 93--94]{gps} we have a factorization 
$ W_{\phi}(g) = \prod\limits_{v} W_{v}(g_{v})$ and 
$W_{f_{\tau'_{s}}}(g) = \prod\limits_{v} W_{f_{\tau'_{s}},v}(g_{v})$ such that 
\begin{equation} \label{euler}
I(\phi,f_{\tau'_{s}}) = \prod_{v} \xi\left( W_{v}, W_{f_{\tau'_{s}},v} \right), 
\end{equation} 
where 
\begin{equation} \label{xi}
\xi\left( W_{v}, W_{f_{\tau'_{s}},v} \right) 
= \int_{\U_{\ell}(k_{v}) \backslash \H(k_{v})} \int_{\X(k_{v})} 
W_{v}(g_{v} x_{v}) W_{f_{\tau'_{s}},v}(g_{v}) dx dg. 
\end{equation}

\subsection{The Unramified Computations}

In this section we compute the zeta integral at the unramified places. This local 
analysis allows us to connect our integrals on the right hand side of (\ref{xi}) 
to local $L$-functions.  

Let $\pi = \otimes_{v} \pi_{v}$ and $\tau'_{s} = \otimes \tau'_{s,v}$ be 
as before. Recall that $\tau'_{s}$ is as in (\ref{tauprime}) in terms of a 
representation $\tau$ of $\GL(m,\A)$ and an id\`ele class character $\omega$ 
of $\GL(1,\A)$. 

\begin{thm} \label{unram}
Let $v$ be a non-archimedean place of $k$ such that $\pi_{v}$ and $\tau'_{s,v}$ are 
unramified. 
If $W^{0}_{\pi_{v}}$ and $W^{0}_{\tau'_{s,v}}$ are class one functions for the local, 
unramified representations $\pi_{v}$ and $\tau'_{s,v}$ respectively, then 
\[ \int\limits_{\X(k_{v})} \quad \int\limits_{\U_{\ell}(k_{v}) \backslash \H(k_{v})} 
W^{0}_{\pi_{v}}(gx) W^{0}_{\tau'_{s,v}}(g) dg dx 
= \begin{cases}
\frac{L(s,\pi_{v} \times \tau_{v})}{L(2s,\tau_{v},\wedge^{2}\otimes\omega^{-1})}, & \mbox{if }\G=\GSpin(2n+1),\\
&\\
\frac{L(s,\pi_{v} \times \tau_{v})}{L(2s,\tau_{v},\sym^{2}\otimes\omega^{-1})}, & \mbox{if }\G=\GSpin(2n) .
\end{cases} \]
\end{thm}

\begin{proof}
The proof uses a decreasing induction on the $\GL$ rank and is completely similar to the proof of 
\cite[Theorem B]{ginzburg}. The starting step of the induction, as in Ginzburg's theorem, is similar 
to the case of $\GL(n)\times\SO(2n+1)$ in \cite{gps}. One uses the Casselman-Shalika formula \cite{cs} 
for the calculation of the Whittaker functions. 
\end{proof} 
 
\subsection{Global Zeta Integral and Partial $L$-functions} 

We now state the major consequence of the above discussions in the global setting we need. 
We use the notation of the earlier sections. 

\begin{thm} \label{rs-Lqt}
Let $\pi$ be a unitary, cuspidal, globally generic, automorphic representation of $\G(n,\A)$ and 
let $\tau$ be a unitary, cuspidal representation of $\GL(m,\A)$. Assume that $m \le n$. 
Moreover, fix an id\`ele class character $\omega$ of $\GL(1,\A)$. For a sufficiently large finite set of places 
$S$, including all the archimedean places and the places where ramification occurs, we 
have 
\begin{eqnarray} \label{Lqt} 
I\left( \phi,f_{\tau'_{s}} \right) 
= \begin{cases}
\frac{L^{S}(s,\pi \times \tau)}{L^{S}(2s,\tau,\wedge^{2}\otimes\omega^{-1})} \cdot R(s), 
& \mbox{if }\G=\GSpin(2n+1),\\
&\\
\frac{L^{S}(s,\pi \times \tau)}{L^{S}(2s,\tau,\sym^{2}\otimes\omega^{-1})} \cdot R(s), 
& \mbox{if }\G=\GSpin(2n),
\end{cases} 
\end{eqnarray}
where $R(s)$ is a meromorphic function, which can be made holomorphic and nonzero in 
a neighborhood of any given $s=s_{0}$ for an appropriate choice of $f$. 
\end{thm}

\begin{proof}
The theorem follows from Theorem \ref{unram} if we set $R(s)$ to be equal 
to the product of the local zeta integrals (\ref{xi}) over $v \not \in S$. The fact that 
$R(s)$ is meromorphic is clear. To show that it can be made holomorphic in 
the neighborhood of any point $s=s_{0}$ the argument is completely similar 
to the one for the case of $\GL(m) \times \SO(2n+1)$ in \cite[\S\S 6-7]{soudry-mem}. 
\end{proof}

\section{The Transferred Representation}

In this section $k$ will continue to denote a number field and $\A=\A_k$ will denote its ring of ad\`eles. 
Let $\G=\G(n)$ denote $\GSpin(2n+1)$, the split $\GSpin(2n)$, or the 
quasi-split $\GSpin^*(2n)$ associated with a quadratic extension $K/k$ of number fields. We will 
refer to the case of $\G(n)=\GSpin(2n+1)$ as the odd case and the remaining cases as the even 
case. 

\subsection{The Global Transfer}

Let $\pi$ be a irreducible, generic, unitary, cuspidal, automorphic representation of $\G(\A)$. Let $\Pi$ 
be a transfer of $\pi$ to $\GL(2n,\A_k)$ as in \cite[Theorem 1.1]{duke} and
Theorem \ref{q-split-weak}. By the classification of automorphic representations of general linear 
groups \cite{js1,js2} we know that $\Pi$ is a constituent of some automorphic representation 
\begin{equation}\label{Sigma} 
\Sigma = \mbox{Ind}\left( |\det|^{r_1} \sigma_1 \otimes \cdots \otimes |\det|^{r_t} \sigma_t\right) 
\end{equation}
with $\sigma_i$ a unitary, cuspidal, automorphic representation of $\GL(n_i,\A)$, $r_i\in\R$, and 
$n_1+n_2+\cdots+n_t = 2n$.  

Let $\omega = \omega_\pi$ denote the central character of $\pi$. Then $\omega$ is a unitary id\`ele class  
character of $k$ and we have shown that $\Pi$ is nearly equivalent to $\w{\Pi} \otimes \omega$. 

Our first goal in this section is to prove the fact that all the exponents $r_i=0$ in (\ref{Sigma}).  
In order to do so, we follow the method of Gelbart, Ginzburg, Piatetski-Shapiro, Rallis and Soudry 
as explained for classical groups in \cite[\S 1]{soudry}. We explained in Section \ref{grs} 
how to generalize parts of this theory to the cases of odd and even $\GSpin$ groups.

We start with a lemma about twisted exterior and symmetric (partial) 
$L$-functions. For its proof we need a result on holomorphy of twisted $L$-function 
in the half plane $\Re(s) > 1$. This result and much more are the subject of two 
works currently being completed, one by Dustin Belt in his thesis at Purdue University, 
and the other by Suichiro Takeda which has appeared as a preprint. 

\begin{prop}\label{twisted-prop}(\cite{belt} and \cite{takeda})  
Let $\chi$ be an arbitrary id\`ele class character and let $\tau$ be a unitary, cuspidal, 
automorphic representation of $\GL(m,\A)$. Let $S$ be a finite set of 
places of $k$ containing all the archimedean places and all the non-archimedean 
places at which $\pi$ ramifies. Then the partial twisted $L$-functions 
$L^{S}(s,\pi,\wedge^{2}\otimes\chi)$ and $L^{S}(s,\pi,\sym^{2}\otimes\chi)$ 
are holomorphic in $\Re(s) > 1$. 
\end{prop}

We remark that Jacquet and Shalika proved that $L^{S}(s,\pi,\wedge^{2}\otimes\chi)$ 
has a meromorphic continuation to a half place $\Re(s) > 1 - a$ with $a>0$ depending 
on the representation \cite[\S 8, Theorem 1]{jac-shalika-ann-arbor}. 
Proposition \ref{twisted-prop} in the case of $\wedge^{2}\otimes\omega$ can also be 
dug out of their work. However, D. Belt's results show this for all $s$, with possible poles at $s=0,1$. 

As far as we know, an analogue of Jacquet-Shalika's result for twisted symmetric square 
was not available. For $m=3$ it follows from results of W. Banks \cite{banks} following 
the untwisted ($\chi=1$) results of Bump and Ginzburg \cite{bump-ginzburg-annals}. 
S. Takeda's results build on this line of work. 

\begin{lem} \label{twsited-L}
Let $m$ be a positive integer and let $\tau$ be an irreducible, unitary, cuspidal, 
automorphic representation of $\GL(m,\A)$. Let $\omega$ be an id\`ele class character 
and let $s \in \C$. Let $S$ be a finite set of places of $k$ including all the archimedean 
ones such the data is unramified outside $S$.  
\begin{itemize} 
\item[(a)] Both $L^S(s, \tau, \wedge^2 \otimes \omega^{-1})$ and 
$L^S(s, \tau, \sym^2 \otimes \omega^{-1})$ 
are holomorphic and non-vanishing for $\Re(s) > 1$. 
\item[(b)] 
If either of the above $L$-functions has a pole at $s=1$, then 
$\tau \cong \w{\tau} \otimes \omega$.  
\end{itemize}
\end{lem}
 
 \begin{proof}
We have 
\[ L^S(s, \tau \otimes (\tau \otimes \omega^{-1})) 
=  L^S(s, \tau, \wedge^2 \otimes \omega^{-1}) L^S(s, \tau, \sym^2 \otimes \omega^{-1}). \] 

The left hand side is holomorphic and non-vanishing for $\Re(s) > 1$ by 
\cite[Proposition (3.6)]{js2}.  Moreover, by Proposition \ref{twisted-prop} both of 
the $L$-functions on the right hand side are holomorphic for $\Re(s) > 1$. Therefore, both 
are non-vanishing there, as well. This is part (a). 

On the other hand, by \cite[Theorem 1.1]{shahidi-pac} both $L$-functions on 
the right hand side are non-vanishing on $\Re(s) = 1$. If one has a pole at $s=1$, 
then the left hand side must have a pole at $s=1$. Again by \cite[Proposition (3.6)]{js2} 
the two representations $\tau$ and $\tau \otimes \omega^{-1}$ must be contragredient of 
each other, i.e., $\tau \cong \w{\tau} \otimes \omega$. This is part (b). 
 \end{proof}

\begin{prop} \label{EisPole}
Let $\tau$ be an irreducible, unitary, cuspidal representation of $\GL(m,\A)$ 
and let $\omega$ be an id\`ele class character. 
Fix $s_0 \in \C$ with $\Re(s_0) \ge 1$ and assume that 
the Eisenstein series $E(g, f_{\tau'_{s}})$ introduced in (\ref{EisConstruct}) has a pole at $s=s_0$.  
Then, $s_0=1$ and $L^S(s, \tau, \wedge^2 \otimes \omega^{-1}) $ has a simple pole at $s=1$ 
in the odd case while $L^S(s, \tau, \sym^2 \otimes \omega^{-1}) $ has a simple pole at $s=1$ 
in the even case. 
\end{prop}

\begin{proof} 
We know from the general theory of Euler products of Langlands and 
the Langlands-Shahidi method that the poles of $E(g, f_{\tau'_{s}})$ 
come from its constant term along $\P_{m}$.  

For a decomposable section $f_{\tau'_{s}}$ the constant term of 
$E(g, f_{\tau'_{s}})$ along $\P_{m}$ has the form 
\begin{equation}
f_{\tau'_{s}}(I) + \prod_{v \in T} M(f^{(v)}_{\tau'_{s}}) 
\frac{L^{T}(2s-1,\tau,\wedge^{2}\otimes \omega^{-1})}
{L^{T}(2s,\tau,\wedge^{2}\otimes \omega^{-1})}
\end{equation} 
in the odd case, and 
\begin{equation} 
f_{\tau'_{s}}(I) + \prod_{v \in T} M(f^{(v)}_{\tau'_{s}}) 
\frac{L^{T}(2s-1,\tau,\sym^{2}\otimes \omega^{-1})}
{L^{T}(2s,\tau,\sym^{2}\otimes \omega^{-1})}
\end{equation} 
in the even case, where $T$ is a finite set of places of $k$ containing $S$.

We should recall that in the constructing the Eisenstein series we used $s-\frac{1}{2}$ 
in (\ref{tauprime}) instead of the usual $s$. This is responsible for the appearance 
of $2s-1$ and $2s$ instead of the usual $2s$ and $2s+1$ in the constant term. 
Furthermore, the terms $2s$ and $2s+1$ appear because we have 
\begin{equation}
\operatorname{Ind}_{\P_{m}(\A)}^{\H(\A)} \left(\tau|\det|^{s}\otimes\omega \right) 
= \operatorname{Ind}_{\P_{m}(\A)}^{\H(\A)} \left(2 s \tilde{\alpha}, \tau\otimes\omega \right),  
\end{equation}
where the right hand side is as in (\ref{induced}) and $\tilde{\alpha}$ on the left hand side is 
the notation from the Langlands-Shahidi method.

The terms $M(f^{(v)}_{\tau'_{s}})$, the local intertwining operators at $I$, 
are holomorphic for $\Re(s) \ge 1$ for all $v$ \cite{shahidi1988-annals,shahidi1990-annals}. 
Therefore, if $E(g, f_{\tau,s})$ has a pole at $s=s_{0}$, then 
\begin{equation} \label{odd-Eis}
\frac{L^T(2s-1, \tau, \wedge^2 \otimes \omega^{-1})} 
{L^T(2s, \tau, \wedge^2 \otimes \omega^{-1})}
\end{equation} 
has a pole at $s=s_0$ in the odd case, or 
\begin{equation} \label{even-Eis}
\frac{L^T(2s-1, \tau, \sym^2 \otimes \omega^{-1})} 
{L^T(2s, \tau, \sym^2 \otimes \omega^{-1})}
\end{equation} 
has a pole at $s=s_0$ in the even case.

Now assume that $E(\cdot, f_{\tau,s})$ does have a pole at $s=s_0$ with $\Re(s_0) \ge 1$. Then 
$\Re(2 s_0) \ge 2$ and by \cite[Prop. 7.3]{kim-shahidi-jrms} the denominator in both 
(\ref{odd-Eis}) and (\ref{even-Eis}) is non-vanishing for $\Re(s) \ge 1$.  
Therefore, the numerator has a pole at $s=s_0$. 
Because $\Re(2 s_{0} -1) \ge 1$ Lemma \ref{twsited-L} implies that $s_{0}=1$ 
and the proof is complete. 
\end{proof}

\begin{thm} \label{rs-hol}
Let $\pi$ be an irreducible, unitary, cuspidal, globally generic representation of $\G(n,\A)$.  
Let $\tau$ be an irreducible, unitary, cuspidal representation of $\GL(m,\A)$ with 
$2 \le m \le n$. 
Assume that $S$ is a sufficiently large finite set of places including all the archimedean 
places of $k$. 

\begin{itemize}

\item[(a)] The $L$-function $L^S(s,\pi \times \tau)$ is holomorphic for $\Re(s) > 1$. 

\item[(b)] Let $\omega$ be an id\`ele class character. 
Assume that $\tau \cong \w{\tau} \otimes \omega$. 
If $L^S(s,\sigma \times \tau)$ has a pole at $s=1$, then 
$L^S(s, \tau, \wedge^2 \otimes \omega^{-1}) $ has a pole at $s=1$ in the odd case 
and $L^S(s, \tau, \sym^2 \otimes \omega^{-1}) $ has a pole at $s=1$ in the even case. 
Such a pole would be simple. 
\end{itemize}
When $m=1$, the $L$-function $L^S(s,\sigma \times \tau)$ is entire in both cases. 
\end{thm}

\begin{proof} 
Assume that $L^S(s,\sigma \times \tau)$ has a pole at $s=s_0$ with $\Re(s_0) \ge 1$. 
By \cite[Prop. 7.3]{kim-shahidi-jrms} we know that both
$L^S(s,\tau, \wedge^2 \otimes \omega^{-1})$ and 
$L^S(s,\tau, \sym^2 \otimes \omega^{-1})$ are  
holomorphic and non-vanishing at $s=2 s_0$. Hence, 
the right hand side of (\ref{Lqt}) has a pole at $s=s_0$. 
Theorem \ref{rs-Lqt} then implies that $I(\phi,f_{\tau'_{s}})$ has a pole at $s=s_{0}$. 
Here $\tau'_{s}$ is defined in terms of $\tau$ and $\omega$ as in (\ref{tauprime}). 

Consequently, the Eisenstein series $E(g,f_{\tau'_{s}})$ must have a pole at $s=s_{0}$. 
Proposition \ref{EisPole} now implies that $s_{0}=1$ and 
$L^S(s, \tau, \wedge^2 \otimes \omega^{-1}) $, in the odd case, or  
$L^S(s, \tau, \sym^2 \otimes \omega^{-1})$, in the even case, 
has a simple pole at $s=1$. This proves (a) and (b). 

Finally, if $m=1$, then the left hand side of (\ref{Lqt}) is entire, which implies that 
$L^S(s,\sigma \times \tau)$ is entire, too. This completes the proof. 
\end{proof}

\begin{thm} \label{image}

Let $\pi$ be an irreducible, automorphic, unitary, cuspidal, globally generic representation of 
$\G(n,\A)$ with central character $\omega=\omega_\pi$ 
and let $\Pi$ be a transfer of $\pi$ to $\GL(2n,\A)$. Assume that $\Pi$ is a subquotient of 
$\Sigma$ as in (\ref{Sigma}) with $n_1 + n_2 + \cdots + n_t = 2n$. 
\begin{itemize}
\item[(a)] We have $r_1 = r_2 = \cdots = r_t = 0.$   
\item[(b)] The representations $\sigma_i$ are pairwise inequivalent, 
$n_i \ge 2$, 
and $\sigma_i \cong \w{\sigma}_i \otimes \omega$ for all $i$. Moreover, 
for $S$ a sufficiently large finite set of places including all the archimedean ones, we have that 
$L^S(s, \sigma_i, \wedge^2 \otimes \omega^{-1}) $ has a pole at $s=1$ in the odd case, and 
$L^S(s, \sigma_i, \sym^2 \otimes \omega^{-1}) $ has a pole at $s=1$ in the even case. 
\end{itemize}
\end{thm}

\begin{proof}
By \cite[Prop. 7.4]{duke} we know that $\Sigma$ is induced from a representation 
of a Levi subgroup of $\GL(2n)$ of type $(a_1,\dots,a_q,b_{1}\cdots,b_{\ell},a_{q},\cdots,a_{1})$ 
which can be written as 
\begin{eqnarray*} 
&\delta_1 |\det(\cdot)|^{z_1} \otimes \cdots \otimes \delta_q |\det(\cdot)|^{z_q} 
\otimes 
\\
&\sigma_1 \otimes \sigma_2 \otimes \cdots \otimes \sigma_\ell 
\\ 
&\otimes (\w{\delta}_q \otimes \omega^{-1}) |\det(\cdot)|^{- z_q} \otimes \cdots 
(\w{\delta}_1 \otimes \omega^{-1}) |\det(\cdot)|^{- z_1}, 
\end{eqnarray*}
where $\delta_i$ and $\sigma_i$ are irreducible, unitary, cuspidal presentations, 
$\sigma_i \cong \w{\sigma}_i \otimes \omega$, and 
\begin{equation} \label{ab} 
2 (a_{1}+\cdots+a_{q}) + (b_{1}+\cdots+b_{\ell}) = 2n. 
\end{equation}

Assume that $q>0$. Rearranging if necessary we may assume  
$\Re(z_1) \le \cdots \le \Re(z_q) < 0$. Now for $S$ a sufficiently large finite set of places we have 
\begin{eqnarray} \label{Lprod}
L^S(s, \pi \times \w{\delta}_1) &=& L^S(s, \Pi \times \w{\delta}_1) \\ \nonumber
&=& \prod_{i=1}^q L^S(s+z_i, \delta_i \times \w{\delta}_1) 
L^S(s-z_i, \w{\delta}_i \times \w{\delta}_1 \otimes \omega^{-1}) \\ \nonumber 
&& \cdot \prod_{i=1}^\ell L^S(s, \sigma_i \times \w{\delta}_1). 
\end{eqnarray}
The first term on the right hand side has a pole at $s=1-z_1$ which can not be canceled by the other terms 
because $\Re(1-z_1 \pm z_i) \ge 1$ and $\Re(1- z_1) \ge 1$. Therefore, the left hand side has a pole 
at $s=1-z_1$. 

On the other hand, by (\ref{ab}) we know that $a_{1} \le n$. We can 
apply Theorem \ref{rs-hol}(a) to conclude that $\Re(z_1) \ge 0.$ 
This is a contradiction proving that $q=0$, i.e., there are no $\delta_i$'s. 

So far we have proved that $\Sigma$ is induced from a representation of the form 
$\sigma_1 \otimes \sigma_2 \otimes \cdots \otimes \sigma_\ell$ satisfying 
$\sigma_i \cong \w{\sigma}_i \otimes \omega$. 
Fix $1 \le j \le \ell$ and consider 
\begin{eqnarray}\label{SigProd}
L^S(s, \pi \times \w{\sigma}_j) = L^S(s, \Pi \times \w{\sigma}_j) 
= \prod_{i=1}^\ell L^S(s, \sigma_i \times \w{\sigma}_j). 
\end{eqnarray}

The right hand side has a pole and hence, so does the left hand side. Moreover, 
the $\sigma_i$'s are pairwise inequivalent because otherwise the left hand 
side of (\ref{SigProd}) would have a pole of higher order. 
Since $\sigma_i \cong \w{\sigma}_i \otimes \omega$ we can apply Theorem \ref{rs-hol}(b). 
We conclude that $L^S(s, \sigma_i, \wedge^2 \otimes \omega^{-1}) $ has a pole at $s=1$ 
in the odd case, and $L^S(s, \sigma_i, \sym^2 \otimes \omega^{-1}) $ has a pole at $s=1$ 
in the even case. 

Finally, Theorem \ref{rs-hol}(c) shows that no $\sigma_j$ is a character.  This completes the proof. 
\end{proof}

\begin{cor}
The representation $\Sigma$ is irreducible and 
$\Pi=\Sigma = \sigma_1 \boxplus \cdots \boxplus \sigma_t$ is an isobaric sum of the $\sigma_i$. 
In particular, the transfer $\Pi$ of $\pi$ is unique and $\Pi \cong \w{\Pi} \otimes \omega$ (not just 
nearly equivalent as in \cite[Theorem 1.1]{duke}). 
\end{cor}
\begin{proof}
The corollary immediately follows from the fact that $r_1=\cdots=r_t = 0$ and that 
$\sigma_i \cong \w{\sigma}_i \otimes \omega$. 
\end{proof}

\subsection{Description of the Image of Transfer}

We continue to denote by $\pi$ an irreducible, globally generic, unitary, cuspidal automorphic representation 
of $\G(n,\A)$. We proved that $\pi$ has a unique transfer $\Pi$, an irreducible, generic, automorphic 
representation of $\GL(2n,\A)$. Moreover, we have shown that $\omega_\Pi = \omega^n\mu$ and 
$\Pi \cong \w{\Pi}\otimes\omega$, where 
$\omega=\omega_\pi$ denotes the central character of $\pi$ and $\omega_\Pi$ denotes that of $\Pi$. 

Furthermore, Theorem \ref{image} gives an ``upper bound'' for the image of transfer from $\G(n)$ groups 
to $\GL(2n)$. Combining this with the ``lower bound'' provided by Hundley and Sayag in 
\cite{hs-announce,hs-ppt-odd,hs-ppt-even} 
gives the full description of the image of this transfer. We summarize all these results as follows. 

\begin{thm}\label{main} 
Let $k$ be a number field and let $\A=\A_{k}$ be the ring of ad\`eles of $k$. 
Denote by $\G(n)$ the split groups $\GSpin(2n+1)$, $\GSpin(2n)$, or any of the quasi-split 
non-split groups $\GSpin^{*}(2n)$. Let $\pi$ be a globally generic, irreducible, cuspidal, automorphic 
representation of $\G(n,\A)$ with central character $\omega=\omega_{\pi}$. 
Then $\pi$ has a unique functorial transfer to an 
automorphic representation $\Pi$ of $\GL(2n,\A)$ associated with the $L$-homomorphism 
$\iota$ described \cite{duke} (split case) and Section \ref{q-split-sec} (quasi-split non-split). 
The transfer $\Pi$ satisfies 
\[ \Pi \cong \w{\Pi}\otimes\omega. \] 
Moreover, 
\[ \omega_{\Pi} = \omega_{\pi}^{n} \mu, \] 
where $\mu$ is a quadratic id\`ele class character which is trivial in the split case and nontrivial 
in the quasi-split non-split case. (The triviality or nontriviality of $\mu = \omega_{\Pi} \omega^{-n}$ 
can tell apart the split and quasi-split non-split cases.)  

The automorphic representation $\Pi$ is an isobaric sum of the form 
\[ \Pi = \operatorname{Ind}\left(\Pi_{1} \otimes \cdots \otimes \Pi_{t} \right) 
= \Pi_{1} \boxplus \cdots \boxplus \Pi_{t},\]
where each $\Pi_{i}$ is a unitary, cuspidal representation of $\GL(n_{i},\A)$ such that for 
$T$ sufficiently large finite set of places of $k$ containing the archimedean places, the partial 
$L$-function   
$L^{T}(s,\Pi_{i},\wedge^{2}\otimes\omega)$ has a pole at $s=1$ in the odd case and 
$L^{T}(s,\Pi_{i},\sym^{2}\otimes\omega)$ has a pole at $s=1$ in the even case (both split 
and quasi-split non-split cases).  We have $\Pi_{i} \not\cong \Pi_{j}$ if $i \not= j$ and 
$n_{1}+\cdots+n_{t}=2n$ with each $n_{i}> 1$. 

Furthermore, any such representation $\Pi$ is a functorial transfer of some globally generic 
$\pi$. 
\end{thm}

\section{Applications} 

\subsection{Local Representations at the Ramified Places} 
The local components of the automorphic representation $\Pi=\otimes_{v} \Pi_{v}$ are 
well understood for the archimedean $v$ as well as those non-archimedean $v$ outside 
of the finite set $S$ through our construction of the candidate transfer. However, the 
converse theorem tells us nothing about $\Pi_{v}$ for $v \in S$. Having proved Theorem \ref{main} 
we can now get some information for these places as well. This shows that while we did not 
have control over places $v \in S$, the automorphic representation $\Pi$ does indeed turn 
out to have the right local components in $S$.

\begin{prop}\label{vinS}
Let $S$ be the non-empty finite set of non-archimedean places as in \cite[Thm. 1.1.]{duke} or 
Theorem\ref{q-split-weak}.  Fix $v \in S$ and let 
\[ \pi_{v} \cong \operatorname{Ind}\left(\pi_{1,v}|\det|^{b_{1,v}}\otimes\cdots\otimes
\pi_{r,v}|\det|^{b_{r,v}}\otimes\pi_{0,v} \right) \] 
be an irreducible, generic representation of $\G(n,k_{v})$, where each $\pi_{i,v}$ is a tempered 
representation of $\GL(n_{i},k_{v})$, $b_{1,v} > \cdots > b_{r,v}$ and $\pi_{0,v}$ is a tempered, 
generic representation of some smaller $\G(m,k_{v})$ with $n_{1}+\cdots+n_{r}+m=n$. 
Denote the central character of $\pi_{v}$ by $\omega_{v}$. 

Assume that $\pi_{v}$ is the local component of the globally generic representation $\pi$ of 
$\G(n,\A_{k})$ and let $\Pi$ be its transfer to $\GL(2n,\A_{k})$. Then the local component $\Pi_{v}$ 
of $\Pi$ at $v$ is generic and of the form 
\begin{eqnarray}\label{vinS-rep} 
\Pi_{v} & \cong & \operatorname{Ind} \left( 
(\pi_{1,v}|\det|^{b_{1,v}}\otimes\cdots\otimes\pi_{r,v}|\det|^{b_{r,v}}\otimes\Pi_{0,v}\otimes \right. \\
&& \left. (\w{\pi}_{r,v}\otimes\omega_{v})|\det|^{-b_{r,v}}\otimes\cdots\otimes(\w{\pi}_{1,v}
\otimes\omega_{v})|\det|^{-b_{1,v}}
\right), 
\end{eqnarray}
where $\Pi_{0,v}$ is a tempered representation of $\GL(2m,k_{v})$ if $m>0$. 
\end{prop}

\begin{proof}
The argument proceeds the same way as in the proof of \cite[Prop. 2.5.]{compos}, which proved 
the analogous result for the case of $\GSp(4)=\GSpin(5)$. We briefly mention the steps for completeness. 

The fact that $\Pi$ is an isobaric sum of unitary, cuspidal representations of general linear 
groups, Theorem \ref{main}, implies that every local component of $\Pi$ is full induced and generic. 
In particular, so is $\Pi_{v}.$ 

The first step is to show that 
\[ \gamma(s,\pi_{v}\times\rho_{v},\psi_{v}) = \gamma(s,\Pi_{v}\times\rho_{v},\psi_{v}) \] 
for every supercuspidal representation $\rho_{v}$ of $\GL(a,k_{v})$. To do this we ``embed'' 
the local representation $\rho_{v}$ in a unitary, cuspidal representation $\rho$ of $\GL(a,\A)$ 
whose other local components are unramified \cite[Prop. 5.1]{shahidi1990-annals} and apply 
the converse theorem with $S'=S - \{v\}$ just as in \cite[Prop. 7.2]{ckpss2}. Moreover, by 
multiplicativity of the $\gamma$-factors we obtain the equality for $\rho_{v}$ in the 
discrete series, as well. 

Next, assume that $\pi_{v}$ is tempered. We claim that $\Pi_{v}$ is also tempered. 
Here again the main tool is multiplicativity of the $\gamma$-factors and the proof is 
exactly as in \cite[Lemma 7.1]{ckpss2}. This proves the Proposition for $r=0$.  

Now consider the case of $r>0$. Apply the case of $r=0$ to $\pi_{v}=\pi_{0,v}$ and take 
the resulting tempered representation of $\GL(2m,k_{v})$ to be $\Pi_{0,v}$. To show that 
this representation satisfies the requirements of the proposition we use the converse theorem 
again. Let $T=\{w\}$ consist of a single non-archimedean place $w \not= v$ at which $\pi_{v}$ 
is unramified and consider the global representation $\Pi'$ of $\GL(2n,\A)$ whose local 
components are the same as those of $\Pi$ outside of $S$ and are the irreducible, induced 
representations on the right hand side of (\ref{vinS-rep}) when $v \in S$. We can apply the converse 
theorem, Theorem \ref{converse}, to $\Pi'$ and $T$ because the induced representations on 
the right hand side of (\ref{vinS-rep}) have the right local $L$-functions. The conclusion is 
that $\Pi'$ is a transfer of $\pi$ (outside of $T$) and by the uniqueness of the transfer, 
Theorem \ref{main}, we have $\Pi'_{v} \cong \Pi_{v}$ for $v \in S$. This completes the proof. 
\end{proof}

\subsection{Ramanujan Estimates}
Following \cite{ckpss2}, we introduce the following notation. Let $\Pi=\otimes_{v} \Pi_{v}$ 
be a unitary, cuspidal, automorphic representation of $\GL(m,\A_{k})$. For each 
place $v$ the representation $\Pi_{v}$ is unitary generic and can be written as a full induced 
representation 
\[ \Pi_{v} \cong \operatorname{Ind}\left(\Pi_{1,v}|\det|^{a_{1,v}}\otimes\cdots\otimes
\Pi_{r,v}|\det|^{a_{r,v}} \right) \]
with $a_{1,v} > \cdots > a_{r,v}$ and each $\Pi_{i,v}$ tempered. 

\begin{defi}
We say $\Pi$ satisfies $H(\theta_{m})$ with $\theta_{m} \ge 0$ if for all places $v$ we 
have $-\theta_{m} \le a_{i,v} \le \theta_{m}$. 
\end{defi}

The classification of the generic unitary dual of $\GL(m)$, \cite{tadic,vogan}, trivially 
gives $H(1/2)$. The best result currently known for a general number field is 
$\theta_{m} = 1/2 - 1/(m^{2}+1)$ proved in \cite{lrs} with a few better results known for small 
values of $m$ and over $\Q$. The Ramanujan conjecture for $\GL(m)$ demands $H(0)$. 

Similarly, if $\pi=\otimes_{v} \pi_{v}$ is a unitary, generic, cuspidal, automorphic representation 
of $\G(n,\A_{k})$ each $\pi_{v}$ can be written as a full induced representation 
\[ \pi_{v} \cong \operatorname{Ind}\left(\pi_{1,v}|\det|^{b_{1,v}}\otimes\cdots\otimes
\pi_{r,v}|\det|^{b_{r,v}}\otimes\tau_{v}\right) , \]
where each $\pi_{i,v}$ is a tempered representation of some $\GL(n_{i},k_{v})$ and $\tau_{v}$ 
is a tempered, generic representation of some $\G(m,k_{v})$ with $n_{1}+\cdots+n_{t}+m=n$. 

\begin{defi}
We say $\pi$ satisfies $H(\theta_{n})$ with $\theta_{n} \ge 0$ if for all places $v$ we 
have $-\theta_{m} \le b_{i,v} \le \theta_{m}$. 
\end{defi}

Again, we would have the bound $H(1)$ trivially as a consequence of the classification of the 
generic unitary dual and the Ramanujan conjecture demands $H(0)$. 

\begin{prop}\label{theta}
Let $k$ be a number field and assume that all the unitary, cuspidal representations of 
$\GL(m,\A_{k})$ satisfy $H(\theta_{m})$ for $2\le m \le 2n$ and 
$\theta_{2} \le \theta_{3} \le \cdots \le \theta_{2n}$. Then any globally generic, unitary, 
cuspidal representation $\pi$ of $\G(n,\A_{k})$ satisfies $H(\theta_{2n})$. In fact, if 
$\pi$ transfers to a non-cuspidal representation $\Pi=\Pi_{1}\boxplus\cdots\boxplus\Pi_{t}$, 
then $\pi$ satisfies the possibly better bound of $H(\theta)$ where 
$\theta = \max \{\theta_{n_1}, \theta_{n_2}, \dots, \theta_{n_t}\}$. 
Here, $\Pi_{i}$ is a unitary, cuspidal representation of $\GL(n_{i},\A_{k})$.  
\end{prop}

\begin{proof}
The argument is exactly the same as the proof of \cite[Theorem 3.3]{compos} and we do not 
repeat it here. Note that our Proposition \ref{vinS} is used for the ramified non-archimedean 
places. 
\end{proof}

\begin{cor}
Every globally generic, unitary, cuspidal, automorphic representation $\pi$ of $\G(n,\A_{k})$ 
satisfy 
\[ H\left(\frac{4n^{2}-1}{2(4n^{2}+1)}\right). \] 
If $\pi$ transfers to a non-cuspidal, automorphic 
representation $\Pi=\Pi_{1}\boxplus\cdots\boxplus\Pi_{t}$, then we can replace $n$ with 
the size of the largest $\GL$ block appearing, resulting in a better estimate.  
\end{cor}

\begin{proof}
This is immediate if we combine Proposition \ref{theta} with the $\GL(m)$ estimate of 
$1/2 - 1/(m^{2}+1)$. 
\end{proof}

We should remark that for small values of $n$ it is possible to obtain better 
estimates because much better estimates are available for small general linear groups 
(and also for $k=\Q$). 
For an example, see \cite[\S 3.1]{compos}

\begin{cor}
The Ramanujan conjecture for the unitary, cuspidal representations of $\GL(m,\A_{k})$ for 
$m \le 2n$ implies the Ramanujan conjecture for the generic spectrum of $\G(n,\A_{k})$. 
\end{cor}

\begin{proof}
This is an immediate corollary of Proposition \ref{theta} where all the $\theta$'s are zero.  
\end{proof}

\subsection{Image of Kim's exterior square} 

H. Kim proved the exterior square transfer of automorphic representations from $\GL(4,\A_{k})$ 
to $\GL(6, \A_{k})$ \cite{kim-jams,henniart-ext2}.  A. Raghuram and the first author gave a complete 
cuspidality criterion for this transfer, determining when the image of this transfer is not cuspidal \cite{ar-ext2}. 
A natural question about the image of this transfer is which automorphic representations of $\GL(6,\A_{k})$ 
are indeed in the image of this transfer. We can now answer this question as an application of our 
Theorem \ref{main}. 

\begin{prop}\label{ext2}  
Let $\Pi$ be a cuspidal, automorphic representation of $\GL(6,\A_{k})$. Then there is a globally 
generic, cuspidal, automorphic representation $\pi$ of $\GL(4,\A_{k})$ such that $\Pi = \wedge^{2} \pi$ 
if and only if there is an id\`ele class character $\omega$ such that the partial $L$-function 
$L^{S}(s,\Pi, \sym^{2}\otimes\omega^{-1})$ has a pole at $s=1$ for $S$ a sufficiently large finite set of 
places of $k$ including all the archimedean ones. 
\end{prop}

\begin{proof}
The proposition follows immediately from our Theorem \ref{main} if we recall that Kim's exterior square 
transfer from $\GL(4)$ to $\GL(6)$ is a special case of the transfer in the split even case of our theorem 
when $m=3$, i.e., the transfer from $\GSpin(6)$ to $\GL(6)$ \cite[Prop. 7.6]{duke}. 

If we assume that $\Pi$ is the transfer of $\pi$, then we have proved that we can take $\omega = \omega_{\pi}$, 
the central character of $\pi$. The opposite direction requires the descent method in our cases and would follow 
from J. Hundley and E. Sayag's ``lower bound'' result for our transfer \cite{hs-announce,hs-ppt-odd,hs-ppt-even} 
because Kim's $\wedge^{2}$ is a special case of transfer from $\GSpin(6)$ to $\GL(6)$ as mentioned above.  
\end{proof}

Another natural question regarding the image of Kim's exterior square transfer is to determine ``the fiber'' 
for each cuspidal $\Pi$ which is indeed in the image. In other words, determine all representations $\pi$ 
such that $\Pi=\wedge^{2}\pi$. 

A further interesting question would be to explore possible overlaps between various transfers to cuspidal 
representations of $\GL(6)$.  As pointed out in \cite[\S 6]{cpss-clay} (for the untwisted $\omega=1$ case) 
and as it is apparent from our Theorem \ref{main} there can be no overlap between the images of transfers 
from $\GSpin(7)$ and quasi-split forms of $\GSpin(6)$ (which includes Kim's transfer) to $\GL(6)$. However, 
there may be potential overlaps with the transfer from unitary groups or the 
Kim-Shahidi transfer \cite{kim-shahidi-annals} 
from $\GL(2) \times \GL(3)$ to $\GL(6)$.


\begin{thebibliography}{10000}

\bibitem[Ar]{arthur} 
	J. Arthur. 
	The Endoscopic Classification of Representations: Orthogonal and Symplectic Groups. 
	Colloquium Publication Series, AMS. 
	To appear. 

\bibitem[A]{asgari}
	M. Asgari. 
	Local $L$-functions for split spinor groups.  
	Canad. J. Math.  54  (2002),  no. 4, 673--693. 

\bibitem[ACS]{me-cogdell-shahidi}
	M. Asgari, J. Cogdell and F. Shahidi. 
	Local Transfer and Reducibility of Induced Representations of $p$-adic Classical Groups. 
	In preparation. 

\bibitem[AR]{ar-ext2}
	M. Asgari and A. Raghuram. 
	A Cuspidality Criterion for the Exterior Square Transfer of Cusp Forms on $\rm GL(4)$. 
	Clay Mathematics Proceedings (volume in honor of F. Shahidi's sixtieth birthday). 
	To appear. Available at \texttt{arXiv:math.NT/0712.4315} .    

\bibitem[AS1]{duke}
    M. Asgari and F. Shahidi. 
    Generic transfer for general spin groups. 
    Duke Math. J., 132 (2006), no 1, 137--190. 

\bibitem[AS2]{compos}
    M. Asgari and F. Shahidi. 
    Generic transfer from $\rm GSp(4)$ to $\rm GL(4)$.  
    Compos. Math.  142  (2006),  no. 3, 541--550.

\bibitem[Bnk]{banks}
	W. D. Banks. 
	Twisted symmetric-square $L$-functions and the nonexistence of Siegel zeros on ${\rm GL}(3)$.  
	Duke Math. J.  87  (1997),  no. 2, 343--353.

\bibitem[Blt]{belt}
	D. Belt. 
	On the holomorphy of the exterior square $L$-functions. 
	Thesis, Purdue University. 
	
\bibitem[B]{borel-corvallis}
	A. Borel. 
	Automorphic ${L}$-functions. 
	In 
	Automorphic forms, representations and $L$-functions, Part 2 (Corvallis, OR, 1977), 
	Proc. Sympos. Pure Math. 33, 27--61. Amer. Math. Soc., Providence, R.I., 1979.

\bibitem[Bou]{bourbaki} 
	N. Bourbaki. 
	\'{E}l\'ements de math\'ematique. {F}asc. {X}{X}{X}{I}{V}. {G}roupes et alg\`ebres de {L}ie. 
	Hermann, Paris, 1981.

\bibitem[BG]{bump-ginzburg-annals} 
	D. Bump and D. Ginzburg. 
	Symmetric square $L$-functions on ${\rm GL}(r)$.  
	Ann. of Math. (2)  136  (1992),  no. 1, 137--205.

\bibitem[CSh]{cas-sha} 
	W. Casselman and F. Shahidi. 
	On irreducibility of standard modules for generic representations. 	
	Ann. Sci. \'Ecole Norm. Sup. (4) 31 (1998), no. 4, 561--589. 

\bibitem[CS]{cs}
	W. Casselman and J. Shalika. 
	The unramified principal series of $p$-adic groups. II. The Whittaker function. 
	Compositio Math. 41 (1980), no. 2, 207--231. 	

\bibitem[CKPSS1]{ckpss1}
	J. W. Cogdell, H. H. Kim, I. I. Piatetski-Shapiro and F. Shahidi. 
	On lifting from classical groups to ${\rm GL}_N$.  
	Publ. Math. Inst. Hautes \'Etudes Sci.  No. 93  (2001), 5--30.

\bibitem[CKPSS2]{ckpss2}
	J. W. Cogdell, H. H. Kim, I. I. Piatetski-Shapiro and F. Shahidi. 
	Functoriality for the classical groups.  
	Publ. Math. Inst. Hautes \'Etudes Sci.  No. 99  (2004), 163--233.

\bibitem[CPS1]{cps1}
	J. W. Cogdell and I. I. Piatetski-Shapiro. 
	Converse theorems for ${\rm GL}_n$.  
	Inst. Hautes \'Etudes Sci. Publ. Math.  No. 79  (1994), 157--214.

\bibitem[CPS2]{cps2}
	J. W. Cogdell and I. I. Piatetski-Shapiro. 
	Converse theorems for ${\rm GL}_n$. II.  
	J. Reine Angew. Math.  507  (1999), 165--188.

\bibitem[CPSS1]{cpss-stab}
	J. W. Cogdell, I. I. Piatetski-Shapiro and F. Shahidi. 
	Stability of $\gamma$-factors for quasi-split groups. 
	J. Inst. Math. Jussieu 7 (2008), no. 1, 27--66. 

\bibitem[CPSS2]{cpss-clay}
	J. W. Cogdell, I. I. Piatetski-Shapiro and F. Shahidi. 
	Functoriality for the quasi-split classical groups. 
	Clay Mathematics Proceedings (volume in honor of F. Shahidi's sixtieth birthday). 
	To appear. Available at \texttt{http://www.math.ohio-state.edu/$\sim$cogdell/} . 

\bibitem[G]{ginzburg}
	D. Ginzburg. 
	$L$-functions for ${\rm SO}_n\times {\rm GL}_k$.  
	J. Reine Angew. Math.  405  (1990), 156--180.

\bibitem[GPS]{gps}
	S. Gelbart and I. I. Piatetski-Shapiro. 
	$L$-functions for $G \times \rm GL(n)$. 
	in 
	Explicit constructions of automorphic $L$-functions. 
	Lecture Notes in Mathematics, 1254. Springer-Verlag, Berlin, 1987. vi+152 pp.

\bibitem[GS]{gelbart-shahidi}
	S. Gelbart and F. Shahidi. 
	Boundedness of automorphic $L$-functions in vertical strips.  
	J. Amer. Math. Soc.  14  (2001),  no. 1, 79--107

\bibitem[HT]{harris-taylor}
	M. Harris and R. Taylor. 
	The geometry and cohomology of some simple Shimura varieties. 
	With an appendix by Vladimir G. Berkovich. 
	Annals of Mathematics Studies, 151. Princeton University Press, Princeton, NJ, 2001. viii+276 pp.

\bibitem[Hei]{hei-muic2007}
	V. Heiermann and G. Mui\'c. 
	On the standard modules conjecture. 
	Math. Z. 255 (2007), no. 4, 847--853. 
	
\bibitem[HO]{hei-opd-ppt}	
	V. Heiermann and E. Opdam.
	On the tempered L-function conjecture. 
	Preprint, 2009. Available at \texttt{arXiv:math.NT/0908.0699} .
	
\bibitem[H1]{henniart}
	G. Henniart. 
	Une preuve simple des conjectures de Langlands pour ${\rm GL}(n)$ sur un corps $p$-adique. 
	Invent. Math.  139  (2000),  no. 2, 439--455.

\bibitem[H2]{henniart-ext2}
	G. Henniart. 
	Sur la fonctorialit\'e, pour $\rm GL(4)$, donn\'ee par le carr\'e ext\'erieur. 
	Mosc. Math. J.  9  (2009),  no. 1, 33--45.

\bibitem[HS1]{hs-announce} 
	J. Hundley and E. Sayag. 
	Descent construction for GSpin groups: main results and applications. 
	Electron. Res. Announc. Math. Sci. 16 (2009), 30--36. 
    
\bibitem[HS2]{hs-ppt-odd}
    J. Hundley and E. Sayag. 
    Descent Construction for GSpin Groups--Odd Case. 
    Preprint. Available at \texttt{http://opensiuc.lib.siu.edu/math\_articles/98/} . 

\bibitem[HS3]{hs-ppt-even}
    J. Hundley and E. Sayag. 
    Descent Construction for GSpin Groups--Even Case. 
    Preprint. Available at \texttt{http://opensiuc.lib.siu.edu/math\_articles/102/} . 

\bibitem[JS1]{js1} 
	H. Jacquet and J. Shalika. 
	On {E}uler products and the classification of automorphic forms. 
	Amer. J. Math., {I.} Amer. J. Math., 103(3):499--558, 1981.

\bibitem[JS2]{js2}
    H. Jacquet and J. Shalika. 
    On {E}uler products and the classification of automorphic forms. 
    Amer. J. Math., {II.} Amer. J. Math., 103(4):777--815, 1981.

\bibitem[JS3]{js}
	H. Jacquet and J. Shalika. 
	A lemma on highly ramified $\epsilon$-factors. 
	Math. Ann. 271 (1985), no. 3, 319--332. 	
		
\bibitem[JS4]{jac-shalika-ann-arbor}
    H. Jacquet and J. Shalika. 
    Exterior square $L$-functions.  
    in 
    Automorphic forms, Shimura varieties, and $L$-functions, Vol. II (Ann Arbor, MI, 1988),  
    143--226, Perspect. Math., 11, Academic Press, Boston, MA, 1990. 
    
\bibitem[JPSS]{jpss-factor}
	H. Jacquet, I. I. Piatetski-Shapiro and J. Shalika. 
	Rankin-Selberg convolutions.  
	Amer. J. Math.  105  (1983),  no. 2, 367--464.

\bibitem[K1]{kim-israelJ}
	H. H. Kim. 
	Langlands-Shahidi method and poles of automorphic {$L$}-functions. {II}.
	Israel J. Math., 117:261--284, 2000.
	Correction: Israel J. Math., 118:379, 2000.
	
\bibitem[K2]{kim-jams}
	H. H. Kim. 
	Functoriality for the exterior square of {${\rm GL}\sb 4$} and the symmetric fourth of {${\rm GL}\sb 2$}.
	With appendix 1 by D. Ramakrishnan and appendix 2 by H. Kim and P. Sarnak.
	J. Amer. Math. Soc., 16(1):139--183 (electronic), 2003.

\bibitem[K3]{kim2005}
	H. H. Kim. 
	On local $L$-functions and normalized intertwining operators.  
	Canad. J. Math.  57  (2005),  no. 3, 535--597.	 

\bibitem[KK]{kim-kim-shahidi-vol} 
	H. H. Kim and W. Kim. 
	On the local $L$-functions and normalized intertwining operators {II}; quasi-split groups. 
	Clay Mathematics Proceedings (volume in honor of F. Shahidi's sixtieth birthday). 
	To appear. 

\bibitem[KS1]{kim-shahidi-annals} 
	H. H. Kim and F. Shahidi. 
	Functorial products for ${\rm GL}_2\times{\rm GL}_3$ and the symmetric cube for ${\rm GL}_2$. 
	With an appendix by Colin J. Bushnell and Guy Henniart.  
	Ann. of Math. (2)  155  (2002),  no. 3, 837--893.

\bibitem[KS2]{kim-shahidi-jrms} 
	H. H. Kim and F. Shahidi. 
	On simplicity of poles of automorphic $L$-functions.  
	J. Ramanujan Math. Soc.  19  (2004),  no. 4, 267--280.
    
\bibitem[WKim]{w-kim}
	W. Kim. 
	Square integrable representations and the standard module conjecture for general spin groups. 
	Canad. J. Math. 61 (2009), no. 3, 617--640. 

\bibitem[KoSh]{kottwitz-shelstad}
	R. Kottwitz and D. Shelstad.
	Foundations of twisted endoscopy.
	Ast\'erisque 255 (1999), 1--190.
	
\bibitem[L1]{langlands}
	R. P. Langlands. 
	Automorphic representations, Shimura varieties, and motives. Ein M\"archen.  
	In 
	Automorphic forms, representations and $L$-functions, Part 2 (Corvallis, OR, 1977), 
	Proc. Sympos. Pure Math. 33, 205--246. Amer. Math. Soc., Providence, R.I., 1979.

\bibitem[L2]{langlands-real} 
	R. P. Langlands.
	{\it On the classification of irreducible representations of real algebraic groups.}
	In 
	Representation theory and harmonic analysis on semisimple Lie groups.  
	Math. Surveys Monogr., 31, 101--170. Amer. Math. Soc., Providence, RI, 1989.

\bibitem[LL]{langlands-labesse}
	J.-P. Labesse and R. P. Langlands. 
	$L$-indistinguishability for ${\rm SL}(2)$. 
	Canad. J. Math. 31 (1979), no. 4, 726--785. 

\bibitem[LRS]{lrs}
	W. Luo and Z. Rudnick and P. Sarnak.
	On the generalized Ramanujan conjecture for ${\rm GL}(n)$.  
	In 
	Automorphic forms, automorphic representations, and arithmetic (Fort Worth, TX, 1996),  301--310, 
	Proc. Sympos. Pure Math., 66, Part 2, Amer. Math. Soc., Providence, RI, 1999.

\bibitem[Mu]{muic}
	G. Mui\'c. 
	A proof of Casselman-Shahidi's conjecture for quasi-split classical groups. 
	Canad. Math. Bull. 44 (2001), no. 3, 298--312. 	

\bibitem[MuSh]{muic-shahidi}
	G. Mui\'c and F. Shahidi. 
	Irreducibility of standard representations for Iwahori-spherical representations.
	Math. Ann. 312 (1998), no. 1, 151--165. 

\bibitem[Sat]{satake}
	I. Satake.
	Theory of spherical functions on reductive algebraic groups over $p$-adic fields. 
	Inst. Hautes \'Etudes Sci. Publ. Math. No. 18 1963 5--69.
	
\bibitem[Sh1]{shahidi:duke85}
	F. Shahidi. 
	Local coefficients as {A}rtin factors for real groups. 
	Duke Math. J., 52(4):973--1007, 1985.
	
\bibitem[Sh2]{shahidi1988-annals}
	F. Shahidi. 
	On the Ramanujan conjecture and finiteness of poles for certain $L$-functions. 
	Ann. of Math. (2)  127  (1988),  no. 3, 547--584. 
	
\bibitem[Sh3]{shahidi1990-annals}
	F. Shahidi. 
	A proof of Langlands' conjecture on Plancherel measures; complementary series for $p$-adic groups.  
	Ann. of Math. (2)  132  (1990),  no. 2, 273--330. 
	
\bibitem[Sh4]{shahidi-ps-vol}
	F. Shahidi. 
	On multiplicativity of local factors. 
	In 
	Festschrift in honor of I. I. Piatetski-Shapiro on the occasion of his sixtieth birthday, 
	Part II (Ramat Aviv, 1989),  279--289, 
	Israel Math. Conf. Proc., 3, Weizmann, Jerusalem, 1990. 

\bibitem[Sh5]{shahidi-pac} 
    F. Shahidi. 
    On non-vanishing of twisted symmetric and exterior square $L$-functions for ${\rm GL}(n)$. 
    Olga Taussky-Todd: in memoriam. 
    Pacific J. Math.  1997,  Special Issue, 311--322.

\bibitem[Sh6]{shahidi-ppt} 
    F. Shahidi. 
    Arthur Packets and the Ramanujan Conjecture. 
    Kyoto J. Math. (memorial issues for the late M. Nagata). 
    To appear. Available at \texttt{arXiv:math.NT/1007.2132} .    

\bibitem[Sou1]{soudry-mem}
	D. Soudry. 
	Rankin-Selberg convolutions for ${\rm SO}_{2l+1}\times{\rm GL}_n$: local theory.  
	Mem. Amer. Math. Soc.  105  (1993),  no. 500, vi+100 pp.
		
\bibitem[Sou2]{soudry} 
    D. Soudry. 
    On Langlands functoriality from classical groups to ${\rm GL}\sb n$. 
    Formes automorphes. I. Ast\'erisque  No. 298 (2005), 335--390. 

\bibitem[Spr]{spr}
	T. A. Springer.
	Linear algebraic groups. Second edition. 
	Progress in Mathematics, 9. Birkh\"auser Boston, Inc., Boston, MA, 1998. xiv+334 pp.  

\bibitem[Td]{tadic}
	Tadi\'c. 
	Classification of unitary representations in irreducible representations of general linear group 
	(non-Archimedean case).  	
	Ann. Sci. \'Ecole Norm. Sup. (4)  19  (1986),  no. 3, 335--382.

\bibitem[Tk]{takeda}
	S. Takeda. 
	The twisted symmetric square $L$-function for $\rm GL(r)$. 
	Preprint. Available at \texttt{arXiv:math.NT/1005.1979} .

\bibitem[T]{tate} 
	J. Tate. 
	Number theoretic background.
	In 
	Automorphic forms, representations and $L$-functions, Part 1 (Corvallis, OR, 1977), 
	Proc. Sympos. Pure Math. 33, 3--26. Amer. Math. Soc., Providence, R.I., 1979.

\bibitem[V]{vogan}
	D. Vogan. 
	Gelfand-{K}irillov dimension for {H}arish-{C}handra modules.
	Invent. Math., 48(1):75--98, 1978.
	    
\end{thebibliography}
\end{document}